\documentclass{amsart}
\usepackage{graphicx} % Required for inserting images

\usepackage{amsfonts}
\usepackage{amssymb}
\usepackage{amsmath}
\usepackage{amsthm}
\usepackage{mathtools}
\usepackage{enumerate, enumitem}
\usepackage{mathrsfs}
\usepackage{xcolor,enumitem}
\usepackage{hyperref}
\newtheorem*{theorem*}{Theorem}
\newtheorem*{question*}{Question}
\newtheorem*{theoremA}{Theorem A}
\newtheorem*{theoremB}{Theorem B}
\newtheorem*{corollaryC}{Corollary C}
\newtheorem*{Goblotstheorem}{Goblot's vanishing theorem}

\newtheorem{theorem}{Theorem}[section]
\newtheorem{proposition}[theorem]{Proposition}

\newtheorem*{proposition*}{Proposition}
\newtheorem{lemma}[theorem]{Lemma}
\newtheorem*{lemma*}{Lemma}

\newtheorem*{corollary*}{Corollary}
\newtheorem{fact}[theorem]{Fact}
\newtheorem*{fact*}{Fact}

\theoremstyle{definition}
\newtheorem{definition}[theorem]{Definition}
\newtheorem*{definition*}{Definition}
\newtheorem{claim}[theorem]{Claim}
\newtheorem*{claim*}{Claim}

\newtheorem*{conjecture*}{Conjecture}

\newtheorem{question}[theorem]{Question}

\theoremstyle{definition}
\newtheorem{example}[theorem]{Example}

\theoremstyle{remark}

\newtheorem*{example*}{Example}

\newtheorem*{remark*}{Remark}

\newtheorem*{note*}{Note}

\newcommand{\Z}{\ensuremath{\mathbb{Z}}}

\newcommand{\A}{\ensuremath{\mathbf{A}}}

\newcommand{\lang} {\ensuremath{\langle} }
\newcommand{\rang}{\ensuremath{ \rangle} }
\newcommand{\lB}{\ensuremath{\lbrace} }
\newcommand{\rB}{\ensuremath{ \rbrace} }

\newcommand{\sbeq}{\ensuremath{ \subseteq }}

\newcommand{\defeq}{\vcentcolon=}

\AtBeginDocument{%
   \def\MR#1{}
}
\begin{document}

\title{Higher limits of wider systems}
\author[J. Bergfalk]{Jeffrey Bergfalk}
\address{Departament de Matem\`{a}tiques i Inform\`{a}tica \\
Universitat de Barcelona \\
Gran Via de les Corts Catalanes 585 \\ 08007 Barcelona, Catalonia}
\email{bergfalk@ub.edu}
\urladdr{https://www.jeffreybergfalk.com/}

\author[M. Casarosa]{Matteo Casarosa}
\address{
Institut de Math\'ematiques de Jussieu - Paris Rive Gauche (IMJ-PRG)\\
Universit\'e Paris Cit\'e\\
B\^atiment Sophie Germain\\
8 Place Aur\'elie Nemours \\ 75013 Paris, France, and Dipartimento di Matematica, Universit\`{a} di Bologna, Piazza di
Porta S. Donato, 5, 40126 Bologna,\ Italy}
\email{matteo.casarosa@imj-prg.fr}
\urladdr{https://webusers.imj-prg.fr/~matteo.casarosa/}

\email{matteo.casarosa@unibo.it}
\urladdr{https://www.unibo.it/sitoweb/matteo.casarosa/en}

\thanks{This work grew out of a 2023 visit, supported by Fondation Sciences Math\'{e}matiques de Paris, of the second author to the University of Barcelona; the first author was additionally supported by Marie Sk\l odowska Curie (CatT 101110452) and Ram\'{o}n y Cajal Fellowships.}

\begin{abstract}
Write $\mathbf{A}_\lambda$ for what might be described as the most elementary nontrivial inverse system of abelian groups indexed by the functions from the cardinal $\lambda$ to the set of natural numbers.
The question of whether for any fixed $n$ the derived limit $\mathrm{lim}^n\,\mathbf{A}_\lambda$ may vanish for only a nonempty subset of the class of infinite cardinals $\lambda$ is recorded in both \cite{Be17} and \cite{Ban23}, and bears closely on several related further ones.
We answer this question in the affirmative; in fact, we show the maximal possibility, namely that this can simultaneously happen in every degree $n>1$.
%, showing, in particular, that the ``governing'' behavior of $\mathrm{lim}^1\,\mathbf{A}_\omega$ (whose vanishing implies that of $\mathrm{lim}^1\,\mathbf{A}_\lambda$ for all $\lambda>\omega$) can simultaneously fail in all degrees $n>1$.
\end{abstract}

\maketitle

\section{Introduction}
\label{sect:intro}

Let $\mathbf{A}_1$ denote the inverse sequence
$$\cdots\to\mathbb{Z}^{n+1}\to\mathbb{Z}^n\to\cdots\to\mathbb{Z}^2\to\mathbb{Z}$$
with bonding maps the natural projections.
More generally, for any cardinal $\lambda$, let $\mathbf{A}_\lambda$ denote the sum, in the pro-category of abelian groups, of $\lambda$ many copies of $\mathbf{A}_1$; its terms are then $\bigoplus_{i<\lambda}\mathbb{Z}^{f(i)}$, as $f$ ranges within the pointwise partial ordering of functions from $\lambda$ to $\mathbb{N}$, with bonding maps again the natural projections.

The higher derived limits of $\mathbf{A}_1$ are all easily seen to vanish, hence for any $n>0$, any nonvanishing $\mathrm{lim}^n\,\mathbf{A}_\lambda$ witnesses the failure of the functor $\mathrm{lim}^n:\mathsf{Pro(Ab)}\to\mathsf{Ab}$ to be additive.
This, in turn, carries consequences in more general mathematics, such as those
in homology theory first underscored in \cite{MP88} or those in condensed mathematics more recently surveyed
in \cite{BeLH24}.
In fact it was \cite{MP88} which first established the possibility of such failures; therein it was shown that $\mathrm{lim}^1\,\mathbf{A}_{\aleph_0}\neq 0$ follows from the continuum hypothesis.
Soon thereafter, \cite{DSV89} showed that $\mathrm{lim}^1\,\mathbf{A}_{\aleph_0}=0$ is consistent with the $\mathsf{ZFC}$ axioms as well.
Together, these works touched off a line of set theoretic research into the derived limits of $\mathbf{A}_{\aleph_0}$ (within which the latter is more typically and simply denoted as $\mathbf{A}$) which continues to this day \cite{Tod89,Kam93,Tod98,Be17,Ban23,VV24,BBMT24,Cas24,CLH24,BanAk25,BanNon25}.
Along the way, many of the most fundamental questions about these derived limits have been answered; consistency proofs of ``$\mathrm{lim}^n\,\mathbf{A}=0$ for all $n>0$'', for example, were recorded in both \cite{BeLH21} and \cite{BHLH23}.
The principal questions which remain fall under two broad headings: (i) relations between the vanishing of $\mathrm{lim}^n\,\mathbf{A}$ and the size of the continuum (see \cite[Question 7.1]{BeLH21}), and (ii) extensions of the tools and results referenced above to generalizations of the system $\mathbf{A}$.
The present work falls under the second heading, wherein three main sorts of theorems shape the terrain.

The first of these concern generalizations of $\mathbf{A}$ at the level of its terms and transition maps.
The inverse systems comprising the natural generalizations are termed \emph{$\Omega$-systems} and, more broadly, \emph{$\Omega_\lambda$-systems} in \cite{BBM23} and \cite{Ban23,BanAk25}, respectively, wherein the fundamental results on $\mathbf{A}$ are shown to extend to these much more general classes.
The second of these concerns generalizations of the ``vertical'' or codomain parameter in $\mathbf{A}$'s index-set to cardinals $\kappa>\aleph_0$.
Here the representative theorem is \cite[Theorem 2.10]{BeLH24}, showing that the associated derived limits are nonzero for any $\kappa\geq\aleph_1$.
Thus it's only for sums of inverse systems of countable height (i.e., of \emph{towers}) that we may hope for positive additivity results, and it is this which has determined our focus on $\mathbf{A}$'s ``horizontal'' generalizations $\mathbf{A}_\lambda$.
This brings us to the area's third main type of theorem \cite[Theorem 5.1]{Be17}:

\begin{theorem*}
$\mathrm{lim}^1\,\mathbf{A}_{\aleph_0}=0$ implies that $\mathrm{lim}^1\,\mathbf{A}_{\lambda}=0$ for every cardinal $\lambda$.
\end{theorem*}
More colloquially, the functor $\mathrm{lim}^1:\mathsf{Pro(Ab)}\to\mathsf{Ab}$ is additive over arbitrary sums of towers if and only if it is additive over countable sums of towers.
The status of this implication in higher degrees has long been the obvious next question; it figures as both \cite[Question 1]{Be17} and \cite[Question 7.2]{Ban23}, and has thus far formed the main avenue of approach to the more general additivity question recorded in our conclusion.
\begin{question*}
Does the above theorem hold for $\mathrm{lim}^n$ in degrees $n>1$?
\end{question*}
In the present work, we show, in two distinct ways, that the answer is no, in the process answering the question's $\Omega_\lambda$-system variants \cite[Question 7.4]{Ban23} and \cite[Question 7.5]{Ban23} as well.
\begin{theoremA}
Assuming the continuum hypothesis, there exists an $\aleph_2$-closed poset $\mathbb{P}$ forcing that $\mathrm{lim}^2\,\mathbf{A}_{\aleph_0}=0$ and $\mathrm{lim}^2\,\mathbf{A}_{\aleph_1}\neq 0$.
\end{theoremA}
This theorem derives additional interest from both the relative simpicity and the combinatorial mainspring of its argument; the latter, in particular, consists in a construction of coherent families satisfying a strengthened version of the classical nontriviality condition of finite-to-one-ness.

We then turn to a stronger theorem, one which forms this paper's main labor.
As background, recall how many of the main results on the derived limits of $\mathbf{A}_\lambda$ and related systems have amounted to essentially linear constructions \cite{MP88,Be17,BeLH21,BBM23,VV24}.
Such approaches, however, are inherently limited in their reach (cf.\ Lemma \ref{lem:no_long_chains} below); accordingly, the development of nonlinear construction techniques has been a particular focus of this paper's second author \cite{Cas24,CLH24}.
The following theorem represents a further advance in these techniques; its key device is the construction of nontrivial $n$-coherent families not along \emph{chains} but along \emph{special $\lambda$-Aronszajn trees} in the partial order $({^\lambda}\omega,\leq)$.
\begin{theoremB}
Assume the generalized continuum hypothesis and that $\Diamond(S_i^{i+1})$ holds for all positive $i<\omega$.
Then for all $n<\omega$ and cardinals $\lambda$, $\lim^{n+1} \A_{\lambda} = 0$ if and only if $\lambda < \aleph_n$.
\end{theoremB}

In plainer English, it is consistent with the \textsf{ZFC} axioms that the groups $\mathrm{lim}^n\,\mathbf{A}_\lambda$ are simultaneously nonzero everywhere that they can be --- or, more precisely, wherever they're not required to be zero by Goblot's vanishing theorem (\cite{Gob70}, although our formulation incorporates the variant \cite[Theorem 2.4]{VV24}).
\begin{Goblotstheorem}
\label{gob}
If the indexing order of a directed inverse system $\mathbf{X}$ of abelian groups is of cofinality $\aleph_k$ for some $k<\omega$ then $\mathrm{lim}^{n}\,\mathbf{X}=0$ for all $n>k+1$; if, in addition, the transition maps of $\mathbf{X}$ are all surjective, then $\mathrm{lim}^{n}\,\mathbf{X}=0$ for $n=k+1$ as well.
\end{Goblotstheorem}
\noindent The transition maps of any $\mathbf{A}_\lambda$ are all surjective and, as is well-known, the premises of our main theorem all hold in, for example, the constructible universe $L$.
Hence:
\begin{corollaryC}
In G\"{o}del's constructible universe, $\mathrm{lim}^n\,\mathbf{A}_\lambda\neq 0$ in every instance not proscribed by Goblot's vanishing theorem.
\end{corollaryC}

This paper's organization is as follows.
We record our most fundamental notations in the course of Section \ref{sect:ntcs+Ak}; further notations are recorded in Section \ref{sect:filtrations} as well.
We prove Theorem A in Section \ref{sect:thmA} and Theorem B in Sections \ref{sect:filtrations} through \ref{sect:thmB}.
Our conclusion then records what figure as the main remaining open questions.

\subsection*{Acknowledgements} We wish to thank Nathaniel Bannister, Tanmay Inamdar, and Alessandro Vignati for valuable discussions of this work.

\section{Nontrivial coherence and the systems $\mathbf{A}_\lambda$}
\label{sect:ntcs+Ak}

Broadly speaking, the study of the derived limits of $\mathbf{A}_\lambda$ and of related systems is a study of \emph{nontrivial coherence}; let us therefore begin by recalling a classical example of the latter.

\begin{example}
\label{ex:fintoone}
There exists a family of functions $\Phi=\langle\varphi_\alpha:\alpha\to\omega\mid\alpha<\omega_1\rangle$ with the following two properties:
\begin{enumerate}
\item $\varphi_\alpha$ and $\varphi_\beta\restriction\alpha$ differ only finitely often for every $\alpha<\beta<\omega_1$;
\item $\varphi_\alpha^{-1}(i)$ is finite for every $\alpha<\omega_1$ and $i<\omega$.
\end{enumerate}
Functions $\varphi_\alpha$ satisfying the weak injectivity condition of item (2) are often called \emph{finite-to-one}, for concision. Since no function $\varphi:\omega_1\to\omega$ is finite-to-one, condition (2) limits the extent of the phenomena of item (1); in particular, any $\Phi$ satisfying (1) and (2) satisfies the following as well:
\begin{enumerate}
\item[(2')] for every $\varphi:\omega_1\to\omega$, $\varphi_\alpha$ and $\varphi\restriction\alpha$ differ infinitely often for some $\alpha<\omega_1$.
\end{enumerate}
Families $\Phi$ satisfying conditions (1) and (2') are called \emph{coherent} and \emph{nontrivial}, respectively; more precisely, they are nontrivially coherent modulo the ideal of finite subsets of $\omega_1$.
In what follows, we denote such \emph{mod finite} relations with an asterisk, recording those as in item (1) as $\varphi_\alpha =^* \varphi_\beta\restriction\alpha$, for instance.

See \cite{Tod07} for canonical examples of such $\Phi$, as well as for their close connection to a range of infinitary combinatorics and to walks on the ordinals in particular.
Such $\Phi$ also witness the nonvanishing both of the first sheaf cohomology group of $\omega_1$ and of $\mathrm{lim}^1$ of the inverse system whose terms are the free abelian groups $\mathbb{Z}^{(\alpha)}$ indexed by $\alpha<\omega_1$ and whose transition maps $\pi_{\alpha\beta}:\mathbb{Z}^{(\beta)}\to\mathbb{Z}^{(\alpha)}$ $(\alpha<\beta<\omega_1)$ are the natural projections.
\end{example}

The nonvanishing of a significant range of derived limit expressions traces to basic variations on the above parameters.
For example, for any cardinal $\lambda$ and function $f:\lambda\to\omega$, write $I(f)$ for the set
$\{(i,j)\in\lambda\times\omega\mid j\leq f(i)\}$
of coordinate pairs falling below the graph of $f$ (i.e., for the \emph{sottografico} of $f$; pl.\ \emph{sottografici}); let also $\leq$ denote the product ordering of such functions, so that $f\leq g$ if and only if $I(f)\subseteq I(g)$.
%We may then posit that:
%For any two such functions $f$ and $g$, let $f\wedge g$ denote their meet in the product ordering of ${^\lambda}\omega$.
\begin{definition}
\label{def:coh_A_lambda}
A family
$$
\left<\varphi_f:I(f)\to\mathbb{Z}\mid f\in{^\lambda}\omega\right>
$$
is \emph{coherent} if
$$\varphi_f=^* \varphi_g\restriction I(f)$$
for all $f\leq g$ in ${^\lambda}\omega$, and is \emph{trivial} if there exists a $\varphi:\lambda\times\omega\to\mathbb{Z}$ such that
$$\varphi_f=^* \varphi\restriction I(f)$$
for all $f$ in ${^\lambda}\omega$.
\end{definition}
Henceforth, we will sometimes omit restriction-notations; sums or comparisons of functions should in all cases be interpreted on the intersections of those functions' domains.
These notions more concretely relate to the systems $\mathbf{A}_\lambda$ as follows.
\begin{proposition}[\cite{Be17}] 
\label{prop:lim1Alambda}
$\mathrm{lim}^1\,\mathbf{A}_\lambda = 0$ if and only if  every coherent family $\langle\varphi_f:I(f)\to\mathbb{Z}\mid f\in{^\lambda}\omega\rangle
$ is trivial.
More precisely, observe that both the collection of coherent such families and its subcollection of trivial such families inherit a group structure from $\mathbb{Z}$; $\mathrm{lim}^1\,\mathbf{A}_\lambda$ is exactly the quotient of the former by the latter.
\end{proposition}
Put differently, the instances of \emph{coherence} and \emph{triviality} appearing in Definition \ref{def:coh_A_lambda} are, respectively, reframings of cocycle and coboundary conditions on the first term of a cochain complex whose $n^{\mathrm{th}}$ cohomology group is the $n^{\mathrm{th}}$ derived limit of $\mathbf{A}_\lambda$, and our analyses of higher $\mathrm{lim}^n\,\mathbf{A}_\lambda$ will center on parallel reframings of higher-dimensional cocycle and coboundary conditions.
Let us turn now to more formal definitions, beginning with that of the system $\mathbf{A}_\lambda$ itself.
\begin{definition}
\label{def:Alambda}
For any cardinal $\lambda$, $\mathbf{A}_\lambda$ denotes the inverse system of abelian groups indexed by $({^\lambda}\omega,\leq)$, with
\begin{itemize}
\item terms $A_f:=\bigoplus_{I(f)}\mathbb{Z}$, and
\item transition maps $\pi_{f,g}:A_g\to A_f$ the natural projections
\end{itemize}
for every $f\leq g$ in ${^\lambda}\omega$.
\end{definition}
We adopt the convention that $0\notin\mathbb{N}$.
For any quasi-order $(X,\trianglelefteq)$ and $n\in\mathbb{N}$, we write $X^{(n)}$ for the collection of $n$-tuples $\vec{x}=(x_0,\dots,x_{n-1})$ in $X$ for which $x_i\trianglelefteq x_j$ for all $i\leq j<n$; for any $j<n$ we then write $\vec{x}^j$ for the $(n-1)$-tuple obtained from $\vec{x}$ by removing $x_i$.
\begin{definition}
\label{def:ncohntriv}
Fix $n\in\mathbb{N}$. If $n=1$ then the family
$$
\Phi=\left<\varphi_{\vec{f}}:I(f_0)\to\mathbb{Z}\mid \vec{f}\in ({^\lambda}\omega)^{(n)}\right>
$$
is \emph{$n$-coherent} or \emph{$n$-trivial}, respectively, if it is coherent or trivial in the sense of Definition \ref{def:coh_A_lambda}.

When $n>1$, the family $\Phi$ is \emph{$n$-coherent} if
$$\sum_{i=0}^{n}(-1)^i\varphi_{\vec{f}^i} =^* 0$$
for all $\vec{f}\in ({^\lambda}\omega)^{(n+1)}$, and is \emph{$n$-trivial} if there exists a family
$$\Psi=\left<\psi_{\vec{f}}:I(f_0)\to\mathbb{Z}\mid \vec{f}\in ({^\lambda}\omega)^{(n-1)}\right>
$$
such that
$$\sum_{i=0}^{n-1}(-1)^i\psi_{\vec{f}^i} =^* \varphi_{\vec{f}}$$
for all $\vec{f}\in ({^\lambda}\omega)^{(n)}$. We write $\partial (\Psi)$ for the $n$-family whose functions are obtained from those of $\Psi$ through the alternating sum on the left-hand side. If these equal modulo finite $\phi_{\vec{f}}$ for every multi-index $\vec{f}$ we write $\partial(\Psi) = \Phi$ and say that $\Psi$ \emph{trivializes} $\Phi$. Sometimes we will apply the multi-index as a subscript to the uppercase Greek letter that stands for a certain family to denote the corresponding function. 
\end{definition}

\begin{proposition}
\label{prop:limnAlambda}
For any cardinal $\lambda$ and $n\in\mathbb{N}$, $\mathrm{lim}^n\,\mathbf{A}_\lambda$ is exactly the quotient of the group of ${^\lambda}\omega$-indexed $n$-coherent families by the group of ${^\lambda}\omega$-indexed $n$-trivial families.
In particular, $\mathrm{lim}^n\,\mathbf{A}_\lambda = 0$ if and only if every such $n$-coherent family is $n$-trivial.
\end{proposition}
See Section 2 of \cite{BeLH21}, \cite{BHLH23}, or of \cite{CLH24} for fuller discussions of these characterizations.
We conclude this section by observing that analyses of the above relations tend to involve analyses of their restrictions to subsystems; a few last points will therefore be important for us below:
\begin{itemize}
\item The relations together described by Definitions \ref{def:Alambda} and \ref{def:ncohntriv} and Proposition \ref{prop:limnAlambda} all relativize in the obvious manner to any directed suborder of $({^\lambda}\omega,\leq)$.
\item More particularly, to define an $n$-coherent family, it will frequently suffice to define it over only a cofinal subset of indices: the essential point here is that such definitions then admit natural $n$-coherent extensions to the full index-set, and that the latter are $n$-trivial if and only if the cofinal subfamily is.
%This may be viewed as a manifestation of $\mathrm{lim}^n$ taking $\mathsf{Pro(Ab)}$ as its domain;
For further discussion, see \cite[Props.\ 2.4 and 2.5]{CLH24}.
\item It is frequently convenient to consider, in place of those of Definition \ref{def:ncohntriv}, families indexed by $({^\lambda}\omega,\leq^*)^{(n)}$, where the ${^*}$ again signals the \emph{with only finitely many exceptions} modification of a coordinatewise relation.
The $n$-coherence and nontriviality of such families are defined exactly as before, and bear, broadly, the same relation to our other objects of interest (and in particular to the vanishing of $\mathrm{lim}^n\,\mathbf{A}$) as those of Definition \ref{def:ncohntriv}; see \cite[\S 3]{Cas24}, for example, for further discussion.
\item If $\mathcal{F}$ is a family of functions and $X$ is a set, then $\mathcal{F}\restriction\restriction X$ denotes the family determined by restricting each function in $\mathcal{F}$ to the intersection of its domain with $X$.
\end{itemize}

\section{Theorem A}
\label{sect:thmA}

In this section, we will show that the following partial order satisfies the conclusions of our introduction's Theorem A.

\begin{definition}
Let $\mathbb{P}$ denote the forcing poset with
\begin{itemize}
\item \textbf{conditions} $2$-coherent families $$p=\left< \varphi^p_{fg}:I(f)\to\mathbb{Z}\mid f\leq g\text{ in } E_p\right>$$
with $E_p$ a $\leq$-directed subset of ${^{\omega_1}}\omega$ of cardinality at most $\aleph_1$,
\item \textbf{ordered by} $p\leq q$ if and only if $p\supseteq q$.
\end{itemize}
\end{definition}
\begin{proof}[Proof of Theorem A]
By assumption, the continuum hypothesis ($\mathsf{CH}$) holds in our ground model $V$; hence by Goblot's vanishing theorem, $\mathrm{lim}^2\,\mathbf{A}_{\aleph_0}=0$ does as well.
The poset $\mathbb{P}$ is $\aleph_2$-closed in the strong sense that, for any $\alpha<\omega_2$, the union of any descending chain $\langle p_i\mid i<\alpha\rangle$ of conditions forms a lower bound for all of them.
$\mathbb{P}$ therefore adds no functions from $\omega_1$ to $\omega$ and, in particular, preserves both $\mathsf{CH}$ and its consequence $\mathrm{lim}^2\,\mathbf{A}_{\aleph_0}=0$.
It is also straightforward to see, again via Goblot's vanishing theorem, that for any $f:\omega_1\to\omega$ the set $\{p\in\mathbb{P}\mid f\in E_p\}$ is dense in $\mathbb{P}$.
Thus the union of a $\mathbb{P}$-generic filter $G$ is, in $V[G]$, a $2$-coherent family of the form $\Phi=\langle \varphi_{fg}:I(f)\to\mathbb{Z}\mid f\leq g\text{ in }{^{\omega_1}}\omega\rangle$, and only the following remains to be seen:
\begin{claim}
\label{clm:1}
$\Phi=\cup G$ is nontrivial in $V[G]$; in particular, $V[G]\vDash\mathrm{lim}^2\,\mathbf{A}_{\aleph_0}\neq 0$.
\end{claim}
We will argue our claim by way of the lemma just below, which we prove in this section's conclusion.
First, fix the following notations:
for any $f:\omega_1\to\omega$ and $\alpha<\omega_1$,
define $f^\alpha:\omega_1\to\omega$ by
$$f^\alpha(\xi)=\begin{cases} f(\xi) & \xi < \alpha \\ 0 & \xi\geq\alpha \end{cases}$$
and let $B=\{h\in{^{\omega_1}}\omega\mid h = f^\alpha\text{ for some }f\in{^{\omega_1}}\omega\text{ and }\alpha<\omega_1\}$.
Put differently, $B$ is the set of functions $h:\omega_1\to\omega$ of bounded support; under the continuum hypothesis, it is of size $\aleph_1$.
\begin{lemma}
\label{lem:1}
For any $E\subseteq{^{\omega_1}}\omega$ of cardinality at most $\aleph_1$ there exists a $g:\omega_1\to\omega$ and nontrivial coherent family $\Theta = \langle \theta_\alpha:I(g^\alpha)\to\mathbb{Z}\mid\alpha<\omega_1\rangle$ such that $\Theta\restriction\restriction I(f)$ is trivial for every $f\in E$.
\end{lemma}
Let us see how the lemma implies Claim \ref{clm:1}, thereby completing the proof of Theorem A.
Suppose for contradiction that the claim is false, and hence that $$p\Vdash\mathrm{``}\dot{\Psi}=\langle \dot{\psi}_f:I(f)\to\mathbb{Z}\mid f\in {^{\omega_1}}\omega\rangle\text{ trivializes }\dot{\Phi}\textnormal{''}$$
for some $p\in G$.
As $\mathbb{P}$ is (strongly) $\aleph_2$-closed, we may extend $p$ to a $q$ with $B\subseteq E_q$ and $E_q$ $\wedge$-closed which decides the functions $\dot{\psi}_f$ for all $f\in E_q$.
The contradiction then takes the form of an $r\leq q$ forcing that $\dot{\Psi}$ does not trivialize $\dot{\Phi}$; more precisely, $r$ will force that (dropping the dot notation and working in $V[G]$) for some $g\in E_r$, the family
$$\left< \varphi_{g^\alpha g}+\psi_{g^\alpha}\mid\alpha<\omega_1\right>$$
is nontrivially coherent.
This, though, implies that no value of $\psi_g$ can satisfy $\psi_g\restriction I(g^\alpha) =^* \varphi_{g^\alpha g}+\psi_{g^\alpha}$ (or, equivalently, $\psi_g\restriction I(g^\alpha) - \psi_{g^\alpha} =^* \varphi_{g^\alpha g}$) for all $\alpha<\omega_1$, and thus that $\Psi$ does not trivialize $\Phi$.
It is Lemma \ref{lem:1} that ensures the existence of such an $r$: to apply it, take $E_q$ for the $E$ of the lemma's premise; then, taking $g$ and $\Theta$ as in its conclusion, let $\varphi^r_{g^\alpha g}=\theta_\alpha - \psi_{g^\alpha}$ for all $\alpha<\omega_1$.
It remains then only to extend these assignments, and those of $q$, to a $2$-coherent family $r$ indexed by a $\leq$-directed $E_r\supseteq E_q\cup\{g\}$ of cardinality $\aleph_1$. %{\color{blue} [Q: can we invoke an abstract principle?]}.
In fact $E_r$ will be the $\vee$-closure of $E_q\cup\{g\}$, and we will argue below that it essentially suffices, by our assumptions on $E_q$, to extend our assignments in a $2$-coherent manner to $\varphi^r_{f,f\vee g}$ and $\varphi^r_{g,f\vee g}$ for all $f\in E_q$.
%We will do so by way of the indices $f\wedge g$, but these may be thought of as playing an essentially auxiliary role.

To this end, well-order $E_q$ and consider at stage $f\in E_q$ the family of equations
\begin{equation}
\label{eq:1}
\varphi_{g^\alpha,(f\vee g)^\alpha} + \varphi_{(f\vee g)^\alpha, f\vee g} - \varphi_{g^\alpha,g} - \varphi_{g, f\vee g}=^* 0
\end{equation}
as $\alpha$ ranges below $\omega_1$.
For readability, we've suppressed notation of the restriction of these functions to the intersection of their domains; note that these are equations which any $2$-coherent family of functions over such indices must satisfy.
Note also that we've already assigned values and relations to these equations' first and third terms, so that the equations (\ref{eq:1}) unpack as
\begin{align*}
\varphi_{g, f\vee g} =^* & \; \varphi_{g^\alpha,(f\vee g)^\alpha} + \varphi_{(f\vee g)^\alpha, f\vee g} - \varphi_{g^\alpha,g} \\ =^* & \; \psi_{(f\vee g)^\alpha} - \psi_{g^\alpha} + \varphi_{(f\vee g)^\alpha, f\vee g} - (\theta_\alpha - \psi_{g^\alpha}) \\ =^* & \; \psi_{(f\vee g)^\alpha} - \theta_\alpha + \varphi_{(f\vee g)^\alpha, f\vee g}
\end{align*}
for every $\alpha<\omega_1$.
Assume that $f\vee g \neq f'\vee g$ for any $f'$ appearing earlier in the well-ordering of $E_q$ (if this is false then the above requirements have already been addressed).
The last line then consists for each $\alpha$ of two pregiven terms together with a term, $\varphi_{(f\vee g)^\alpha, f\vee g}$, which we are free to define.
We must do so so that as $\alpha$ ranges, the resulting expressions form, over the domain $I(g)$, a coherent family trivialized by $\varphi_{g, f\vee g}$.
More particularly, our choice must offset both the nontriviality of $\Theta$ and the potential noncoherence of $\langle \psi_{(f\vee g)^\alpha}\mid\alpha<\omega_1\rangle$; the natural choices are thus given by
\begin{equation}
\label{eq:2}
\varphi^r_{(f\vee g)^\alpha,f\vee g}(\xi,i)=\begin{cases} \theta_\alpha(\xi,i) - \psi_{(f\vee g)^\alpha}(\xi,i) & \text{ if }(\xi,i)\in I(g^\alpha) \\ - \psi_{(f\vee g)^\alpha}(\xi,i) & \text{ if }(\xi,i)\in I((f\vee g)^\alpha)\backslash I(g^\alpha) \end{cases}
\end{equation}
These together determine $\varphi_{g, f\vee g}$ up to finite error: the constantly zero function is now the natural choice for $\varphi^r_{g, f\vee g}$.

Next, consider for the same $f$ the family of equations
\begin{equation}
\label{eq:3}
\varphi_{f^\alpha,(f\vee g)^\alpha} + \varphi_{(f\vee g)^\alpha, f\vee g} - \varphi_{f^\alpha,f} - \varphi_{f, f\vee g}=^* 0
\end{equation}
as $\alpha$ ranges below $\omega_1$.
These unpack much as above, but to a system of equations in some tension with the procedure just described.
For we now have that
\begin{align*}
\varphi_{f, f\vee g} =^* & \; \varphi_{f^\alpha,(f\vee g)^\alpha} + \varphi_{(f\vee g)^\alpha, f\vee g} - \varphi_{f^\alpha,f} \\ =^* & \; \psi_{(f\vee g)^\alpha} - \psi_{f^{\alpha}} + \varphi_{(f\vee g)^\alpha, f\vee g} - (\psi_f - \psi_{f^\alpha}) \\ =^* & \; \psi_{(f\vee g)^\alpha} + \varphi_{(f\vee g)^\alpha, f\vee g} - \psi_f
\end{align*}
for every $\alpha<\omega_1$, and under the definitions in (\ref{eq:2}), the $\alpha$-dependent terms in the last line sum, over the domain $I(g)$, to the nontrivial coherent family $\Theta$.
Crucially, however, only the portion $I(f\wedge g)$ of that domain is in play in the equations (\ref{eq:3}): Lemma \ref{lem:1} ensures us a trivialization $\upsilon:I(f\wedge g)\to\mathbb{Z}$ of $\Theta\restriction\restriction I(f)$, and thus
\begin{equation*}
\varphi^r_{f,f\vee g}(\xi,i)=\begin{cases} \upsilon(\xi,i) - \psi_{f}(\xi,i) & \text{ if }(\xi,i)\in I(f\wedge g) \\ - \psi_{f}(\xi,i) & \text{ if }(\xi,i)\in I(f)\backslash I(f\wedge g) \end{cases}
\end{equation*}
combines with the other elements of $r$ so far defined to satisfy the equations of (\ref{eq:3}).

Proceeding in this fashion, one defines the functions $\varphi^r_{f, f\vee g}$ and $\varphi^r_{g, f\vee g}$ for all $f \in E_q$.
These determine, in turn, definitions (mod finite) of all yet-undefined elements of $r$: for $h,f\in E_q$, for example, $h\leq f\vee g$ implies that $f\vee g = (f\vee h)\vee g$, and hence that the associated $\varphi^r_{h,f\vee g}$ must (mod finite) equal the sum of the defined terms $\varphi^r_{h,f\vee h}$ and $\varphi^r_{f\vee h, (f\vee h)\vee g}$.
The somewhat tedious but elementary remaining details, and in particular the verification that these assignments indeed $2$-cohere, are left to the reader.
Modulo the proof of Lemma \ref{lem:1}, this completes the proof of Theorem A.
\end{proof}

Before proving Lemma \ref{lem:1}, let us briefly expand on how it's interesting.
For each $n\in\omega$, write $\bar{n}$ for the function $\omega_1\to\omega$ constantly taking the value $n$ and $E$ for the collection $\{\bar{n}\mid n\in\omega\}$.
Next, recalling the $\Phi$ of Example \ref{ex:fintoone}, define a family $\Psi=\langle\psi_\beta:\beta\times\omega\to\mathbb{Z}\mid\beta<\omega_1\rangle$ by
\[
\psi_\beta(\alpha,j)=
\begin{cases} 1 & \text{if }\varphi_{\beta}(\alpha)=j \\ 0 & \text{otherwise}
\end{cases}
\]
Each $\psi_\beta$ is, in other words, the characteristic function of the graph of $\varphi_\beta$, and the finite-to-oneness of $\Phi$ translates as the fact that $\Psi\restriction\restriction I(f)$ is trivial for every $f\in E$ (more precisely, it translates as the fact that each such $\Psi\restriction\restriction I(f)$ is, mod finite, $0$, in the natural sense).
The connection to Lemma \ref{lem:1} should be clear; for a better fit, for any $\alpha<\omega_1$ and $n<\omega$ let $g(\omega\cdot\alpha+n)=n$, and for any $\alpha<\beta<\omega_1$ and $j\leq n<\omega$ let
\[
\theta_{\omega\cdot\beta}(\omega\cdot\alpha+n,j)=
\begin{cases} 1 & \text{if }j=n=\varphi_{\beta}(\alpha) \\ 0 & \text{otherwise}
\end{cases}
\]
This $g$ and the obvious extension of our definition to a $$\Theta=\langle\theta_\gamma:I(g\restriction\gamma)\to\mathbb{Z}\mid\gamma<\omega_1\rangle$$ are then as the lemma envisions for the collection $E=\{\bar{n}\mid n\in\omega\}$.\footnote{Note a slight imprecision here: the lemma posits, in the language most natural to the proof of Theorem A, families of functions $\theta_\gamma$ of domain $I(g^\gamma)$. Domains of $I(g\restriction\gamma)$ are somewhat more natural to lemma's proof, though, and the former so readily derive from the latter (via extension by zeros) that we ignore the difference for the remainder of this section.}
From this perspective, the lemma may be viewed as generalizing the classical anti-triviality condition of finite-to-oneness to larger and more arbitrary collections of functions $E$.
\begin{proof}[Proof of Lemma \ref{lem:1}]
Fix $E=\langle f_\alpha\mid\alpha<\omega_1\rangle$ as in the statement of the lemma; by expanding $E$ if necessary, we may assume that $f_n=\bar{n}$ for each $n\in\omega$.
Fix also a bijection $b[\alpha]:\omega\to\alpha$ for each infinite $\alpha<\omega_1$, and define $g:\omega_1\to\omega$ by
\[
g(\omega\cdot\alpha+n)=
\begin{cases} n+1 & \text{if }\alpha<\omega \\ \mathrm{sup}\{f_{b[\alpha](k)}(n)\mid k\leq n\}+1 & \text{otherwise}
\end{cases}
\]
We define a coherent $\Theta=\langle\theta_\gamma:I(g\restriction\gamma)\to\mathbb{Z}\mid\gamma<\omega_1\rangle$
as desired by recursion on $\beta<\omega_1$ --- or more precisely on $\gamma=\omega\cdot\beta$ for such $\beta$ --- as follows.

At successor steps $\beta = \alpha+1$, define $\theta_\gamma:I(g\restriction\gamma)\to\mathbb{Z}$ for any $\gamma\in (\omega\cdot\alpha,\omega\cdot\beta]$ by
\[
\theta_\gamma(i,j)=
\begin{cases} \theta_{\omega\cdot\alpha}(i,j) & \text{if }i<\omega\cdot\alpha \\ 1 & \text{if }i=\omega\cdot\alpha\text{ and }j=g(i) \\ 0 & \text{otherwise}
\end{cases}
\]
Note that this step conserves the condition that, on any full subdomain of the form $I(g\restriction [\omega\cdot\eta,\omega\cdot\eta+\omega))$, the support of any $\theta_\gamma$ constructed so far is of size exactly $1$.
We will denote the unique element $(i,j)$ of this support by $s_\gamma(\eta)$.

At limit steps $\beta$, let $\xi_\ell = \max \{b[\beta](k)\mid k\leq\ell\}$ and let $C_\beta =\{\xi_\ell\mid\ell<\omega\}$.
% Actually, any $C_\beta$ would do.
Fix next an injection $h:\beta\to C_\beta$ so that $h(\alpha)>\alpha$ for all $\alpha\in\beta$.
Define $\hat{\theta}_{\omega\cdot\beta}:I(g\restriction\omega\cdot\beta)\to\mathbb{Z}$ by
$$\hat{\theta}_{\omega\cdot\beta}(\omega\cdot\alpha+n,j)=\theta_{\omega\cdot\mathrm{min}(C_\beta\backslash(\alpha+1))}(\omega\cdot\alpha+n,j)$$
for any $\alpha<\beta$ and $n<\omega$.
This definition conserves both coherence and the support condition noted above.
We will modify it, though, to conserve something deeper, namely the condition that
\begin{equation}
\label{eq:induction_condition}
\tag{RH($\beta$)}
\mathrm{supp}(\theta_\gamma\restriction I(f_\alpha\restriction [\omega\cdot\alpha,\gamma)))\text{ is finite for any }\alpha<\beta\text{ and }\gamma\in(\omega\cdot\alpha,\omega\cdot\beta].
\end{equation}
We do so via finite modifications (thereby conserving coherence) to each of the constituent ``blocks''
$$\hat{\theta}_{\omega\cdot\beta}\restriction I(g\restriction [\omega\cdot\xi_\ell,\omega\cdot\xi_{\ell+1}))=\theta_{\omega\cdot\xi_{\ell+1}}\restriction I(g\restriction [\omega\cdot\xi_\ell,\omega\cdot\xi_{\ell+1}))$$ of $\hat{\theta}_{\omega\cdot\beta}$.
The idea of these modifications is simply, for $\eta\in [\xi_\ell,\xi_{\ell+1})$ with $h(\alpha)\leq\xi_\ell$, to ``shift'' any $s_{\omega\cdot\xi_{\ell+1}}(\eta)$ falling within $I(f_\alpha)$ out of it, along the graph of $g$; by our inductive hypotheses $\mathrm{RH}(\xi_\ell)$, for each $\ell$ this is only necessary for a finite collection $F$ of such $\eta$.
More precisely, for each such $\eta\in F$, let $$n=\max\,b[\eta]^{-1}\{\alpha\mid h(\alpha)\leq\xi_\ell\text{ and }s_{\omega\cdot\xi_{\ell+1}}(\eta)\in I(f_\alpha)\}$$
and define $\theta_{\omega\cdot\beta}$ so that
\[
s_{\omega\cdot\beta}(\eta)=
\begin{cases} (\omega\cdot\eta+n,g(\omega\cdot\eta+n)) & \text{if }\eta\in F \\ s_{\omega\cdot\xi_{\ell+1}}(\eta) & \text{otherwise}
\end{cases}
\]
for any $\eta\in[\xi_\ell,\xi_{\ell+1})$.
This completes our description of the construction of $\Theta$.

To conclude, let us check that $\Theta$ is as desired.
Its coherence was noted above.
Its nontriviality follows from our choice of $\{f_n\mid n\in\omega\}$, which ensure that $\Theta$ exhibits the finite-to-one-ness condition of Example \ref{ex:fintoone} in the manner discussed just before this proof.
And the triviality of $\Theta\restriction\restriction I(f_\alpha)$ for any $f_\alpha\in E$ follows from the fact that for any $\gamma\in (\omega\cdot\alpha,\omega_1)$, the support of $\theta_\gamma\restriction I(f_\alpha\restriction (\omega\cdot\alpha,\gamma))$ is finite.
\end{proof}

\section{Spine and tree filtrations}
\label{sect:filtrations}

We turn now to the proof of Theorem B --- or, more precisely, to its essential preliminaries; as indicated, the theorem's argument involves organizations of the family $({^\lambda}\omega,\leq^*)$ which are, in some sense, unavoidably nonlinear, and the focus of the present section is their specification.

Let us begin by fixing a few further notations. Both the quasi-orders $\leq$ and $\leq^*$ will be in play below; we reiterate that a $^*$ affixed to $=$, $\leq$, or $<$ signals that the relation in question holds on all but finitely many coordinates.
The letters $\kappa$ and $\lambda$ will always denote cardinals; $\mathrm{cf}$ denotes the cofinality function.
$S^\lambda_\kappa$ denotes the stationary set of cofinality-$\kappa$ ordinals below $\lambda$;
we adopt the more readable variant $S_m^n$ when $\kappa = \aleph_m$ and $\lambda = \aleph_n$.
Much as in Section \ref{sect:thmA}, we will have occasion to consider the Euclidean division of an ordinal $\alpha$ by $\omega_n$.
We denote the quotient and the remainder by $q_n(\alpha) $ and $r_n(\alpha)$, so that $r_n(\alpha) < \omega_n$ and $\alpha = \omega_n \cdot q_n(\alpha) + r_n(\alpha)$.
The following definition's a bit more general than it needs to be: in all our applications of it below, $Q$ will itself already be a chain.

%If $(P, C)$ is a strong unbounded pair we denote by $C(\alpha)$ the $\alpha$-th element of the chain and $C\restriction \alpha$ the initial segment of length $\alpha$.

%[I'll refer to the "downward-closure with respect to $\leq^\ast$ of the $\vee$-closure" simply as the ideal (in the order-theoretic sense) generated by a certain set. The symbol is $Cl(\cdot)$.

%$f[X] $ is the direct image of $X$ through $f$. For the range of the entire function we write $ran(f)$.

%We work with increasing filtrations of ideals.

%In this paper, by "continuity" we mean that something is preserved at/by unions.

%$\leq^\ast$ means "mod finite" as always, NOT "mod bounded".]

\begin{definition}
\label{def:spines}
Let $\trianglelefteq$ be a quasi-order on sets $Q\subseteq P$.
We say that \emph{$P$ possesses a $\kappa$-spine $R$ in $Q$} for some cardinal $\kappa$ if there exists a $\vartriangleleft$-increasing enumeration $\langle r_\alpha \mid \alpha<\kappa\rangle$ of a $\trianglelefteq$-cofinal subset $R$ of $Q$.
We say that $(P,Q)$ is a \emph{$\kappa$-unbounded pair} if
\begin{itemize}
\item $P$ is $\trianglelefteq$-directed,
\item $\mathrm{cf}(P)=\mathrm{cf} (Q)=\mathrm{cf}(R)=\kappa$,
\item $Q$ is $\trianglelefteq$-unbounded in $P$, and
\item $P$ possesses a $\kappa$-spine in $Q$.
\end{itemize}
A \emph{$\kappa$-spine filtration} for a $\kappa$-unbounded pair $(P,Q)$ is a continuous filtration of
\[  P = \bigcup_{\alpha < \kappa}   P_\alpha \]
by $\trianglelefteq$-order-ideals $P_\alpha$ together with a $\kappa$-spine $R$ in $Q$ such that $\mathrm{cf}(P_\alpha)<\kappa$ and $\mathrm{otp}(R\cap P_\alpha)=\alpha$ for all $\alpha<\kappa$.
\end{definition}
In our contexts, unbounded pairs admit nice spine filtrations, for the join operation $(f,g)\mapsto f\vee g$ (computed in $({^\lambda}\omega,\leq)$) is a function $j$ as below for both the quasi-orders $\leq$ and $\leq^*$.

\begin{lemma}
\label{increasingunion}

Let $\kappa$ be an uncountable cardinal and let $(P,Q)$ be a $\kappa$-unbounded pair equipped with a function $j: P^2 \to P$ such that:
\begin{enumerate}
\item $x,y\trianglelefteq j(x, y)$ and
\item $j(x,y) \trianglelefteq j(x', y') $
\end{enumerate}
for every $x \trianglelefteq x'$ and $y \trianglelefteq y'$ in $P$.
Then there exists a $\kappa$-spine filtration for $(P, Q)$ such that each $P_\alpha$ is closed with respect to $j$ (and thus, in particular, is directed).
\end{lemma}

\begin{proof}
Observe that $\kappa$ is regular by definition. Fix a spine $R'\subseteq Q$ witnessing that $(P,Q)$ is $\kappa$-unbounded and let $e: \kappa \to P$ enumerate a cofinal subset of $P$.

We construct an $R=\langle r_\alpha\mid\alpha<\kappa\rangle$ and the family $\langle P_\alpha\mid\alpha<\kappa\rangle$ by joint recursion on $\beta<\kappa$: at each stage $\beta$, the set $\langle r_\alpha\mid\alpha<\beta\rangle$ is already defined and we let $P_\beta$ denote the $\trianglelefteq$-downwards closure of the $j$-closure of $\{r_\alpha\mid\alpha<\beta\}\cup\{e(\alpha)\mid\alpha<\beta\}$.
Because this $j$-closure is cofinal in $P_\beta$ (and of size less than the regular cardinal $\kappa$) and for each of its elements $x$ there exists some $r'\in R'$ unbounded by $x$, the set $R'\backslash P_\beta$ is nonempty; we let $r_\beta$ equal its $\trianglelefteq$-minimal element.
This describes the construction, from which the asserted ordertype, continuity, $j$-closure, and cofinality properties of $\langle P_\alpha\mid\alpha<\kappa\rangle$ and $R$ all immediately follow.
\end{proof}

%\color{blue} [JB: commented out: a remark that a stationary subcollection of a spine filtration consists of unbounded pairs. Do we use this?]

%\begin{remark} 
%\label{remarkunbounded}
%Note that, for every limit $\lambda$, $\cof(P_\lambda) \geq \cof(\lambda)$. This follows from the fact that for every limit $\lambda < \omega_{n+1} $,  $C' \cap P_\lambda$ is unbounded in $P_\lambda$. {\color{magenta} [Why is this so? What does seem likely (but do we have any use for it?) is that if $\kappa$ is a regular uncountable cardinal then for any $\kappa$-filtration as in the definition above, the set $$\{\alpha<\kappa\mid (P,R\cap P_\alpha)\textnormal{ is a cf}(\alpha)\textnormal{-unbounded pair}\}$$is a club subset of $\kappa$. Or, we should add this unboundedness condition into the definition?]} Otherwise, for each element in a cofinal set of lesser cardinality we pick one of the chain that is not dominated by it, and taking the supremum we get a contradiction, as in the previous proof.

%So, for $\alpha \in S^k_{k+1}$, $(P_\alpha, P_\alpha \cap C')$ is a $k$-unbounded pair. 
%\end{remark}

%\color{black}

Below, the quasi-orderings $\leq$ and $\leq^*$ of subcollections $P$ of the families ${^\lambda}\omega$ will often be simultaneously significant.

\begin{definition}
Let $(P,Q)$ be a $\kappa$-unbounded pair within the ordering $({^\lambda}\omega,\leq^\ast)$.
If $Q$ is well-ordered and its identity map induces an order-isomorphism of $(Q, \leq)$ with $(Q, \leq^\ast)$ then we call $Q$ a \emph{strong chain} and call $(P,Q)$ a \emph{strong $\kappa$-unbounded pair}. %In other words, each function is everywhere weakly above the previous ones and infinitely often strictly above them.  
\end{definition}

%\color{blue} simply my way of saying that all the $\leq^\ast$s in the chain are actually $\leq$s. \color{black}

%Note that if   $(\lang P_\alpha \rang_{\alpha < \omega_{k+1}}, C')$ is a spine filtration for a strong unbounded pair, then $P_\alpha \cap C' = C' \restriction \alpha$.

For better or worse, however, such substructures are, in our main settings of interest, less extensive than one might wish.

\begin{lemma}
\label{lem:no_long_chains}
If $\lambda$ is an uncountable cardinal and $\mathcal{F}$ is a $\leq^*$-increasing chain in ${^\lambda}\omega$ then the cofinality of $\mathcal{F}$ is at most $2^{<\lambda}$.
In particular, under the generalized continuum hypothesis, for every $n\in\mathbb{N}$ there exists no $\leq^*$-chain in ${^{\omega_n}}\omega$ of ordertype $\omega_{n+1}$; more particularly, in the language of Definition \ref{def:spines}, $({^{\omega_n}}\omega,\leq^*)$ possesses no $\omega_{n+1}$-spine in itself.
\end{lemma}

\begin{proof}
Fix, for contradiction, an $\mathcal{F}'$ witnessing the failure of the lemma and an $<^*$-increasing enumeration of a cofinal subset $\mathcal{F}=\langle f_\xi:\lambda\to\omega\mid\xi<\lambda\rangle$ of $\mathcal{F}'$ which is of regular cardinal ordertype $\mu>2^{<\lambda}$.
For any $\alpha<\lambda$, let $\nu(\alpha)$ denote the $\leq^*$-cofinality of $\mathcal{F}\restriction\restriction\alpha=\langle f_\xi\restriction\alpha:\alpha\to\omega\mid\xi<\lambda\rangle$; in particular, for each $\alpha<\lambda$ there exists a $c_\alpha:\nu(\alpha)\to\mu$ with $\langle f_{c_\alpha(i)}\restriction\alpha\mid i<\nu(\alpha)\rangle$ an increasing cofinal subset of $\mathcal{F}\restriction\restriction\alpha$.
Since $\nu(\alpha)\leq 2^{<\lambda}$, the image of $c_\alpha$ is bounded by some $f\in\mathcal{F}$ and since, therefore, $f_{c_\alpha(i)}\restriction\alpha\leq^* f\restriction\alpha$ for all $i<\nu(\alpha)$, it follows that $\nu(\alpha) = 1$.

Thus we have for each $\alpha<\lambda$ an $f(\alpha)\in\mathcal{F}$ whose $\alpha$-truncation bounds $\mathcal{F}\restriction\restriction\alpha$.
Choose $f\leq^* g$ in $\mathcal{F}$ with both $\leq^*$-above all $f(\alpha)$ $(\alpha<\lambda)$ and a $\beta<\lambda$ such that $f\restriction\beta\not\geq^* g\restriction\beta$.
By definition, though, $g\restriction\beta\leq^* f(\beta)\leq^* f\restriction\beta$, a contradiction.
\end{proof}

%\begin{definition}

%Let $(P, \leq)$ be a quasi-order. We say that $C \sbeq P$ is a $\kappa$-\textit{chain} if there is an injection with cofinal range $\sigma: \kappa \to C$ such that $\alpha \leq \beta$ implies $\sigma(\alpha) \leq \sigma(\beta)$.
    
%\end{definition}

%\begin{definition}

%Let $k \geq 0$. We say that $(P,C)$ is a $k$-unbounded pair if $P$ is an ideal, $\cof(P) = \aleph_k$, and $C \sbeq P$ is an $\omega_k$-chain and unbounded in $P$. 

%\end{definition}

%\begin{definition}
%Let $k\geq 0$, and let $(P,C)$ be a $(k+1)$-unbounded pair. A $(k+1)$-spine filtration for $(P, C)$ is a continuous filtration of ideals

%\[  P = \bigcup_{\alpha < \omega_{k+1}}   P_\alpha \]

%such that for a certain cofinal subchain $C' \sbeq C$ and for all $\alpha < \omega_{k+1} $, $\cof(P_\alpha) \leq \aleph_k$ and $otp(P_\alpha \cap C') = \alpha$.

%\end{definition}

%\begin{remark}

%In case $P= {}^\omega \omega $, one natural choice for the function $d$ is the lowest upper bound $\vee : ({}^{\omega_n} \omega)^2 \to {}^{\omega_n} \omega   $ defined as $(f \vee g)(i) = \max \lB f(i) , g(i) \rB  $. The same is true for every subset $P \sbeq {}^{\omega_n} \omega$ that is $\vee$-closed.

%This condition is minimal in the sense that being closed for $\vee$ and downward-closed for $\leq^\ast$ is equivalent to being an order-theoretic ideal with respect to that same relation.
    
%\end{remark}
By the above combined with the easy fact below (see \cite[Proposition 2.12]{monk2004general}) the generalized continuum hypothesis implies that  the following holds for every $n\in\mathbb{N}$: there exist no $(n+1)$-unbounded pairs $(({}^{\omega_n} \omega , \leq^\ast),Q)$. 
\begin{fact}
If $\lambda$ is an infinite cardinal then $2^{\lambda} = \lambda^+$ implies that $\mathrm{cf}({}^{\lambda} \omega , \leq^\ast) = \lambda^+$.
\end{fact}

The broad implication is that analyses of the derived limits of uncountably wide systems $\mathbf{A}_\lambda$ will tend to involve more general order relations than linear ones.
Underlying our Theorem B, however, is the observation that, even as it constrains the appearance of linear suborders of ${^{\omega_n}\omega}$, the generalized continuum hypothesis furnishes the next best thing, namely $\omega_{n+1}$-Aronszajn tree structures organizing length-$\omega_{n+1}$ filtrations of ${^{\omega_n}}\omega$ by families $\{(P_\alpha,Q_\alpha)\mid\alpha<\omega_{n+1}\}$ of unbounded pairs.
To fix notation, let us briefly recall these structures' basic features.

\begin{definition}
A partial order $(T,\prec)$ is a \emph{tree} if the set $t\!\downarrow:=\{s\in T\mid s \prec t\}$ is well-ordered for any $t\in T$.
The \emph{height} $\mathrm{ht}_T(t)$ of such a $t$ is the ordertype of $t\!\downarrow$.
We write $\mathrm{lev}_\alpha(T)$ for the $\alpha^{\mathrm{th}}$ level $\{t\in T\mid \mathrm{ht}_T(t)=\alpha\}$ of $T$; the height of $T$ is the least $\alpha$ for which $\mathrm{lev}_\alpha(T)=\varnothing$.
For any $n\geq 0$ we write $\Lambda_n$ for the class of multiples of $\omega_n$; for any class $X$ of ordinals, we write $T\restriction X$ for the restriction of $\prec$ to $\{t\in T\mid \mathrm{ht}_T(t)\in X\}$.
For any regular infinite cardinal $\lambda$, a \emph{$\lambda^+$-Aronszajn tree} is a tree $T$ of height $\lambda^+$, with levels all of cardinality at most $\lambda$; such a $T$ is \emph{special} if it admits an $F:T\to\lambda$ for which $F(s)=F(t)$ implies that $s$ is incomparable with $t$.
\end{definition}

The derivation of special Aronszajn trees from cardinal arithmetic hypotheses traces essentially to Specker in 1949.
\begin{theorem}[\cite{Speck49}]
Let $\lambda$ be a regular infinite cardinal.
If $2^{<\lambda}=\lambda$ then there exists a special $\lambda^+$-Aronszajn tree.     
\end{theorem}

The following lemma is this section's main result.
For simplicity, when $n$ is finite, we will tend below to write $n$-\textit{unbounded pair} for what is more properly denoted an \emph{$\aleph_n$-unbounded pair}, and we will similarly write $n$-\textit{spine filtration} for \emph{$\aleph_n$-spine filtration} in the sections which follow.
\begin{lemma}
\label{lemmatree}
Suppose that $n>0$ and $2^{\aleph_n} = \aleph_{n+1}$ and, hence, that there exists a special $\omega_{n+1}$-Aronszajn tree $T$.
There then exists
\begin{enumerate}
\item a length-$\omega_{n+1}$ continuous filtration of 
$^{\omega_n} \omega$
by $\leq^*$-order-ideals $P_\alpha$ each of cofinality at most $\aleph_n$, and
\item an order-preserving function $H : T\restriction\Lambda_n \to (^{\omega_n}\omega,\leq)$
\end{enumerate}
such that for all $\alpha \in S_n^{n+1}$ there exists a $Q_\alpha \subseteq H[T\restriction (\Lambda_n\cap\omega_n\cdot\alpha)]$ for which:
\begin{enumerate}
\setcounter{enumi}{2}
\item $(P_\alpha, Q_\alpha)$ is a strong $n$-unbounded pair, and
\item there exists an $f_\alpha \in P_{\alpha+1} $ such that $g \leq f_\alpha$ for every $g \in Q_\alpha$, and, moreover, every such $g$ falls strictly below $f_\alpha$ on infinitely many coordinates.
\end{enumerate}
\end{lemma}
%Our argument below applies equally well to the $n=0$ case if we restrict to those $\alpha\in S_0^1$ which are limits of limit ordinals, but we won't need this.
\begin{proof}
Fix a bijection $e: \omega_{n+1} \to {}^{\omega_n} \omega$ and a specializing function $F:T \to \omega_n$.
Write $T'$ for $T\restriction \Lambda_n$ (so that $\mathrm{lev}_\alpha(T') = \mathrm{lev}_{\omega_n \cdot \alpha} (T)$) and for $\beta < \omega_n$ let $B_\beta = [\omega \cdot\beta,\omega\cdot\beta + \omega )$.
We will define an order-preserving function $H: T' \to ({}^{\omega_n} \omega, \leq)$ and filtration
\[
{}^{\omega_n} \omega=\bigcup_{\alpha<\omega_{n+1}}P_\alpha
\]
such as the lemma posits by recursion on the ordinals $\alpha<\omega_{n+1}$, conserving the condition that at each stage $\alpha$, each of the following objects is already defined:
\begin{itemize}
\item $H \restriction (T'\restriction\alpha)$, 

\item the $\leq^*$-order-ideal $P_{\alpha}$ generated by $H [T'\restriction\alpha] \cup e[\alpha]$, and

\item functions $G_\xi: \omega_n \to P_{\xi+1}$ with $\leq^*$-cofinal image in $P_{\xi+1}$ for each $\xi<\alpha$.
\end{itemize}

At the stage $\alpha=1$, for example, we may assume that $H$ of the root of $T$ is the constant function $\omega_n\to\{0\}$, so that $P_1$ is the $\leq^*$-order-ideal generated by $e(0)$, and we may let $G_0$ constantly take the value $e(0)$.

At any given stage $\alpha$, we extend the definition of $H$ to $\mathrm{lev}_\alpha(T')$ as follows: for $x\in\mathrm{lev}_\alpha(T')$ and $\gamma\in\omega_n$, let 
\[
H(x)(\gamma) =  \begin{cases} 
G_{q_n(\mathrm{ht}_T(z))} (r_n(\mathrm{ht}_T(z))) (\gamma) +1 & \text{if } \gamma \in B_{F(z)}\text{ for some }z <_T x \\
0 & \textrm{otherwise.} 
\end{cases}
\]
%This is a good definition because it only uses $G_\xi$ for $\xi < \alpha$, and if a $z <_T x$ such that $\gamma \in B_{F(z)}$ exists, then it is unique (by specializingness of the function) and so is the function $G_{q_n(ht_T(z))} (r(ht_T(z))$.
As indicated above, $P_{\alpha+1}$ is then the $\leq^*$-order-ideal generated by $H [T'\restriction\alpha+1]\cup c[\alpha+1]$.
That $T'$ is Aronszajn thus ensures that $P_{\alpha+1}$ is of $\leq^*$-cofinality at most $\omega_n$, and hence that selecting a map $G_\alpha : \omega_n \to P_{\alpha+1}$ with $\leq^\ast$-cofinal image presents no difficulties.
These definitions furnish the material for subsequent stages; this describes the construction.

To see that $H$ is order-preserving, fix $x \leq_{T'} y$ and observe that $\gamma \in \mathrm{supp}(x)$ implies that $\gamma \in B_{F(z)}$ for some $z <_T x \leq_T y$ and thus that $$H(y) (\gamma) = G_{q_n(\mathrm{ht}_T(z))} (r_n(\mathrm{ht}_T(z)))(\gamma)+1 = H(x)(\gamma).$$ 
This verifies the lemma's item (2); item (1) is immediate from our definition of the ideals $P_\alpha$.
Furthermore, by the continuity of this filtration,  if $\alpha<\omega_{n+1}$ is a limit then $\bigcup_{\xi < \alpha} \mathrm{im}(G_\xi)$ is $\leq^*$-cofinal in $P_\alpha$; put differently, any $f$ in this cofinal set is of the form $f = G_\eta (\varepsilon)$ for some $\eta < \alpha$ and $\varepsilon < \omega_n$.
For any $x \in \mathrm{lev}_\alpha(T')$, then, letting $z$ be the unique node below $x$ in $T$ such that $\mathrm{ht}_T(z)=\omega_n \cdot \eta + \varepsilon$, we have that
$$H(x)(\gamma) = G_\eta(\varepsilon)(\gamma)+1 > G_\eta (\varepsilon)(\gamma) = f(\gamma)$$ for each of the infinitely many $\gamma$ in the interval $B_{F(z)}$.
This shows that $H(x) \not \in P_\alpha$, or, more generally, that $H[\mathrm{lev}_\alpha(T')] \cap P_\alpha = \emptyset$ whenever $\alpha<\omega_{n+1}$ is a limit.
It also implies that if $x <_{T'} y$ are both of limit heights then $H(x)$ falls strictly below $H(y)$ on infinitely many coordinates, and a similar argument shows that if $C\subseteq T'$ is a chain of limit-of-limits
%[also, fixed/rephrased the definition of $Q_\alpha$ a little bit] \color{black}
ordertype then $H[C]$ is $\leq^*$-unbounded in $P_\alpha$ for $\alpha=\sup\, \{\mathrm{ht}_{T'}(x)\mid x \in C\}$.
In particular, choosing for each $\alpha \in S_n^{n+1}$ a $y_\alpha \in \mathrm{lev}_\alpha(T')$ and letting $$Q_\alpha = H[\{x\in T'\mid \mathrm{ht}_{T'}(x)\text{ is a limit and }x <_T y_\alpha \}]$$ and letting $f_\alpha = H(y_\alpha)$ defines families of functions as posited in the lemma's items (3) and (4). 
\end{proof}

\section{The continuity and existence of twistable sets}
\label{sect:twistable}

Assume now the hypotheses and notations of Lemma \ref{lemmatree}.
To prove Theorem B, in Sections \ref{sect:nontriviality} and \ref{sect:thmB} we will construct witnesses to $\mathrm{lim}^{n+1}\,\mathbf{A}_{\aleph_n} \neq 0$ which restrict  to an $(n+1)$-coherent nontrivial family on $\mathrm{im}(H)$. This witness is in turn built from $n$-coherent but non-$n$-trivial families
%which moreover will have nontrivial restriction to $H[T \restriction \Lambda_n]$,
on $P_\alpha$ whose restriction to $Q_\alpha$ is nontrivial for every $\alpha$ in $S^{n+1}_n$.
The construction of these $n$-coherent families from guessing principles is much as in \cite{Cas24}; the setting of \emph{wide} inverse systems, however, introduces an additional difficulty.
This is the following: the inductive arguments of Section \ref{sect:nontriviality} require a small list (e.g., a $\lozenge(S^{k+1}_k)$-sequence) of codes for initial fragments of putative trivializations of $(k+1)$-coherent families, as $k$ ranges below $n$.
However, since $\vert I(f)  \vert = \aleph_n$ for any $f \in {^{\omega_n}}\omega$, we cannot even hope to code all functions $I(f)\to\mathbb{Z}$ in a sequence $\langle S_\alpha \subseteq\alpha \mid \alpha < \omega_{k+1}\rangle$.
To address this issue, we introduce the notion of a \emph{$k$-twistable set}, which is subset of the grid $\omega_n \times \omega$ just large enough to support nontrivial $k$-coherence for $(P_\alpha, Q_\alpha)$; within our argument, it will then be sufficient to code putative trivializations on twistable sets.

\begin{definition}
Let $S$ be a set of ordinals. We say that $\alpha$ is a limit point of $S$ if $\alpha \in S$ and $\alpha = \sup(\alpha \cap S)$.
\end{definition}

\begin{definition}

Let $(P,Q)$ be such that $Q \subseteq P \subseteq {}^{\omega_n} \omega$ and $\mathrm{cf}(P) \leq \aleph_0$ with respect to $\leq^\ast$.
We say that $E \subseteq \omega_n \times \omega$ is \emph{$0$-twistable for $(P,Q)$} if $\vert E \vert = \aleph_0$ and $\vert E \cap I(f) \vert < \aleph_0$ for all $f \in P$. 

Next, fix $k \geq 0$ and let $(\langle P_\alpha \mid \alpha < \omega_{k+1} \rangle, Q')$ be a $(k+1)$-spine filtration for some strong $(k+1)$-unbounded pair $(P, Q)$.
We say that \emph{$E$ is $(k+1)$-twistable for $(\langle P_\alpha \mid \alpha < \omega_{k+1} \rangle,  Q')$} if $E$ is the union of a sequence $\mathcal{E}=\langle E_\alpha\mid \alpha \in \{0\}\cup S^{k+1}_k \rangle$ of sets such that

\begin{enumerate}
    \item[(0)] $\vert E_0 \vert \leq \aleph_k$ and $E_0 \subseteq I(Q'(0))$, 
\end{enumerate}

\begin{enumerate}

    \item $E_\alpha \subseteq I(Q'(\alpha))$ and
    
    \item $E_\alpha$ is $k$-twistable for $(P_\alpha, Q' \restriction \alpha)$ 
\end{enumerate}    
    for all $\alpha\in  S^{k+1}_k$, and
\begin{enumerate}

    \item[(3)] $E_\alpha \subseteq \bigcup (\mathcal{E}\restriction\alpha)$ for any limit point $\alpha$ of $S_k^{k+1}$.
\end{enumerate}

\noindent Similarly, we will say that $E$ is $(k+1)$-twistable for a strong unbounded pair $(P,Q)$ if and only if there exists a spine filtration of $(P,Q)$ for which $E$ is $(k+1)$-twistable.
\end{definition}

Twistable sets are small by definition. 

\begin{lemma}
\label{sizetwist}
If $E$ is $k$-twistable, then $\vert E \vert \leq \aleph_k$.
\end{lemma}

\begin{proof}
This is definitional for $k=0$ and immediate for all higher $k$ by induction.
\end{proof}

It's equally clear that any order-ideal of countable cofinality admits a $0$-twistable set; more particularly:

\begin{lemma}
\label{0exist}
Let $P\sbeq {}^{\omega_n} \omega $ be an $\leq^*$-ideal of countable cofinality; suppose also that $g:\omega_n\to\omega$ is not in $P$.
Then there exists a $0$-twistable $E \sbeq I(g)$ for $P$.
\end{lemma}

\begin{proof}
Fix a $\leq^\ast$-cofinal subset $\{f_n\mid n\in\omega\}$ of $P$ and choose an $x_n \in I(g) \setminus \bigcup_{k\leq n} I(f_k)$ for each $n$; clearly, $E = \{ x_n \mid n\in\omega\}$ is as desired.    
\end{proof}

%{\color{blue} [JB: A heritage of lingering uncertainties in Section \ref{sect:filtrations} --- e.g., do we really have any need, in its constructions, for $\leq^*$? Do we really need to make terminological room for the possibility that $Q$ isn't itself a spine (or chain)? (it would be a much cleaner paper if we didn't) --- is that we haven't yet finalized our notion of ``ideal''. Presumably we're going with $\leq^*$, but do we need to? Let's discuss next time.]}\\

The following lemma expresses the fact that the notion of twistable set only depends on the tail of a spine filtration.

%{\color{blue} [JB: If this is the case, then why undertake the ugliness of indexing twistable sets by $\{0\}\cup S_k^{k+1}$ rather than just $S_k^{k+1}$?][MC: this is the case precisely because we can store an initial segment of a twistable set into $E_0$]} 

\begin{lemma}

\label{lemmafinsegm}

Let $k \geq 0$ and suppose that $(P, Q_1)$ and $(P, Q_2)$ are two strong $(k+1)$-unbounded pairs such that there exists $f \in Q_1 \cap Q_2$ such that $\{ g \in Q_1 \mid g \geq f \} = \{ g \in Q_2 \mid g \geq f \}$. Then a set $E$ is $(k+1)$-twistable for $(P, Q_1)$ if and only if it is $(k+1)$-twistable for $(P, Q_2)$.

\end{lemma}

\begin{proof}

As the hypotheses are symmetrical we only show one implication. Also, the proof will be by induction, even though inductive reasoning is only required for one of the conditions to be checked. Call $P(k)$ the statement above relative to a fixed $k \geq 0$.

Let $(\langle P_\alpha \mid \alpha < \omega_{k+1} \rangle, Q')$ be a spine filtration of $(P, Q_1)$ for which $E$ is $(k+1)$-twistable, that is, it admits a decomposition $E = \bigcup \lbrace E_\alpha \mid \alpha \in \lbrace 0 \rbrace \cup S^{k+1}_k \rbrace$ as in the definition. There exists a successor ordinal $\alpha_0$ such that $Q'(\alpha_0) \geq f$ and hence thanks to our assumptions $Q' \restriction (\omega_{k+1} \setminus \alpha_0) \sbeq Q_2 $. So, once we reset the indexing of this final segment by letting $Q^\ast(\beta) = Q'(\alpha_0 + \beta)$ for all $\beta < \omega_{k+1}$ we obtain that  $(\langle P_{\alpha_0 + \beta} \rangle_{\beta < \omega_{k+1}}, Q^\ast )$ is a spine filtration for $(P, Q_2)$. As $\alpha_0$ is a successor, $(\alpha_0 + \beta) \in S_k^{k+1}$ if and only if $\beta \in S_k^{k+1}$. Hence we let $E'_0 = \bigcup \lbrace E_\alpha \mid \alpha \in \lbrace 0 \rbrace \cup (\alpha_0 \cap S_k^{k+1}) \rbrace  $ and $E'_{\beta} = E_{\alpha_0 + \beta}$. Clearly, $\vert E'_0 \vert \leq \aleph_k$ by an easy application of Lemma \ref{sizetwist}, and $E'_0  \sbeq I(Q'(\alpha_0)) = I(Q^\ast (0))$. Now we check conditions $(1)-(3)$. 

$(1)$ is satisfied since $E'_\beta = E_{\alpha_0 + \beta} \sbeq I(Q'(\alpha_0 + \beta)) = I(Q^\ast (\beta))$.

$(3)$ is satisfied since if $\beta = \sup(\beta \cap S_k^{k+1}) $ then $\alpha_0 + \beta = \sup( (\alpha_0 + \beta) \cap S_k^{k+1}) $, hence 
\[E'_\beta = E_{\alpha_0 + \beta} \sbeq \bigcup \lbrace E_\alpha \mid \alpha \in \lbrace 0 \rbrace \cup ((\alpha_0 + \beta) \cap S_k^{k+1}) \rbrace \]  
\[   = (E_0 \cup \bigcup_{\alpha \in \alpha_0 \cap S_k^{k+1}} E_\alpha ) \cup \bigcup_{\beta \in S_k^{k+1}} E_{\alpha_0 + \beta} = E'_0 \cup \bigcup_{\beta \in S_k^{k+1}} E'_\beta \]
Finally, $(2)$ is immediate if $k=0$ and follows easily from $P(k-1)$ otherwise. In fact, what we need to prove is that $E'_\beta = E_{\alpha_0 + \beta} $ is $k$-twistable for  $(P_{\alpha_0 + \beta} , Q^\ast \restriction \beta) = (P_{\alpha+ \beta}, Q' \restriction (\alpha_0 + \beta \setminus \alpha_0) )$, and we know that it is for the $k$-unbounded pair $(P_{\alpha_0 + \beta}, Q' \restriction \alpha_0 + \beta)$ of which that other chain is a final segment. So we conclude by $P(k-1)$.

Finally, notice that $E=  \bigcup \lbrace E'_\beta \mid \lbrace 0 \rbrace \cup \beta \in S_k^{k+1} \rbrace$.
\end{proof}

With the next few lemmas, we show by induction on $k$ that the notion of $k$-twistable set is \textit{continuous}, in the sense that when taking the increasing union of certain sequences of unbounded pairs, the union of the corresponding twistable sets contains a twistable set for that increasing union. The proofs of both the base case and the induction step are based on some form of diagonalization. In the remainder of this section, we work in ${}^{\omega_n} \omega$ for some fixed $n$.

\begin{lemma}

\label{0cont}

Let $P= \bigcup_{i < \omega} P_i$ be an increasing union of ideals of countable cofinality, each with its $0$-twistable set $X_i$. Then $X= \bigcup_{i < \omega} X_i$ contains a $0$-twistable set for $P$.

\end{lemma}

\begin{proof}

Note that since each $P_i$ is $\leq$-directed, for each we have a $\leq^\ast$-cofinal, $\leq$-ordered chain $\langle f_{i, j} \mid j < \omega \rangle$. Using the fact that $P_i \subseteq P_{i'}$ for $i < i'$, we can iteratively find $j(i)$ such that:

\begin{enumerate}
    \item $f_{i, j(i)} \leq f_{i+1, j(i+1)} $

    \item $ f_{l,m} \leq f_{i+1, j(i+1)} $ for all $l,m \leq i$.
\end{enumerate}

Then we consider the diagonal chain $\langle f_{i, j(i)} \mid i < \omega \rangle$ and observe that it is cofinal in $P$. For each $i < \omega$ pick $x_i \in X_i \setminus I(f_{i, j(i)})$. Then $\lB x_i \rB_{i < \omega} $ is almost disjoint from all the $I(f_{i, j})$ and hence from all the \textit{sottografici} of functions in $P$.     
\end{proof}

For $k \geq 0$ we define the statement $R(k)$, which we also call $k$-continuity, as follows. $R(k)$ holds if and only if whenever 

    \begin{enumerate}

    \item $(P,Q)$ is a strong $(k+1)$-unbounded pair. 
    \item$(\lang P_\xi \mid \xi < \omega_{k+1} \rang , Q') $ is a spine filtration of it,
    \item $\zeta$ is a limit point of $S^{k+1}_k$,
    
    \item $(\xi_\alpha)_{\alpha < \omega_k} \sbeq \zeta \cap S_k^{k+1}$ is increasing and cofinal in $\zeta$,
    
    \item For all $\alpha < \omega_k$, there exist a $k$-twistable $E_{\xi_\alpha} \sbeq I(Q'(\xi_\alpha))$ for $(P_{\xi_\alpha}, Q' \restriction \xi_\alpha)$,

    \end{enumerate}

there exists $E_\zeta \sbeq  \bigcup_{\alpha< \omega_k} E_{\xi_\alpha}  \sbeq I(Q'(\zeta))$ that is $k$-twistable for the unbounded pair $(P_\zeta, Q'\restriction \zeta).$

\begin{lemma}
\label{contind}

$R(k)$ implies $R(k+1)$. 
    
\end{lemma}

\begin{proof}

Throughout the proof, let $q_k(\alpha)$ and $r_k(\alpha)$ denote the quotient and the remainder of $\alpha$ in the Euclidean division by $\omega_k$, respectively.

We now assume $R(k)$ and the hypotheses of $R(k+1)$ to prove the conclusion of $R(k+1)$. 

Hypothesis $R(k+1).(5)$ gives us, for each index $\alpha < \omega_{k+1}$, a spine filtration $(\lang P_{\xi_\alpha, \beta} \rang_{\beta < \omega_{k+1}}, Q''_{\xi_\alpha}) $ and $k$-twistable sets $E_{\xi_\alpha, \beta} $ for all $\beta \in S_k^{k+1}$, such that $E_{\xi_\alpha} =  \bigcup \langle E_{\xi_\alpha, \beta} \mid \beta \in \{ 0 \} \cup S_k^{k+1} \rangle$. 

We recursively define a function $F: \omega_{k+1} \to \omega_{k+1} $ such that 
\[  P_{\xi_\alpha, F(\alpha)} \supseteq \bigcup_{\gamma < \alpha} P_{\xi_\gamma, F(\gamma)+ \omega_k +1 } \cup \bigcup_{\epsilon, \eta < \alpha} P_{\xi_\epsilon, \eta}.               \]
This is possible because the right-hand side is a union of at most $ \aleph_k$-many ideals of cofinality at most $ \aleph_k$, and is included into $P_{\xi_\alpha}$.

Then we let
\[ P'_\gamma \defeq P_{\xi_{q_k(\gamma)}, F(q_k(\gamma)) +r_k(\gamma) +1},\]
and finally
\[  P''_\alpha \defeq \bigcup_{\gamma < \alpha} P'_\gamma.  \]
As for the chain, we let
\[  Q''_\zeta ( \alpha ) = Q''_{\xi_{q_k(\alpha)}} (F(q_k(\alpha)) + r_k(\alpha)).     \]

\begin{claim}
\label{claim_filt}
$(\lang P''_\alpha \mid \alpha < \omega_{k+1} \rang, Q''_\zeta) $ is a spine filtration for $(P_\zeta, Q' \restriction \zeta)$.
\end{claim}

\begin{proof}[Proof of Claim \ref{claim_filt}]
The condition satisfied by $F$ implies that $P'_\gamma$ is increasing and hence $P''_\alpha$ is increasing and continuous by construction. Again by the property that defines $F$, we have $\bigcup_{\alpha < \omega_{k+1}} P''_\alpha = P_\zeta$. Finally, for each $\alpha < \omega_{k+1} $, $\mathrm{cf}(P'_\alpha) \leq \aleph_k$, and hence $\mathrm{cf}(P''_\alpha)$ as the latter is the union of at most $\aleph_k$-many ideals of cofinality less or equal than $ \aleph_k$.

As for the chain, we have the inclusions $Q''_\zeta \sbeq \bigcup_{\alpha< \omega_{k+1}} Q''_{\xi_\alpha} \sbeq Q' \restriction \zeta$, following from $Q''_{\xi_\alpha} \sbeq Q' \restriction \xi_\alpha \sbeq Q' \restriction \zeta$. It remains to prove that, for every $\alpha < \omega_{k+1}$, $Q''_\zeta \restriction \alpha = P''_\alpha \cap Q''_\zeta $, or equivalently that $P''_\alpha \not \ni Q''_\zeta(\alpha) \in P''_{\alpha+1}$. This holds because for every $\gamma < \alpha $, we have

\[ Q''_\zeta(\alpha) = Q''_{\xi_{q_k(\alpha)}} (F(q_k(\alpha)) +r_k(\alpha)) \not \in P'_\gamma = P_{\xi_{q_k(\gamma)}, F(q_k(\gamma)) + r_k(\gamma)+1}.\] In fact, if $q_k(\gamma) < q_k(\alpha)$,then, by construction of $F$, 

\begin{align*}
Q''_\zeta (\alpha) 
&= Q''_{\xi_{q_k(\alpha)}} \left( F(q_k(\alpha)) + r_k(\alpha) \right) \\
&\not\in P_{\xi_{q_k(\alpha)}, F(q_k(\alpha)) + r_k(\alpha)} \\
&\supseteq P_{\xi_{q_k(\alpha)}, F(q_k(\alpha))} \\
&\supseteq P_{\xi_{q_k(\gamma)}, F(q_k(\gamma)) + \omega_k + 1} \\
&\supseteq P_{\xi_{q_k(\gamma)}, F(q_k(\gamma)) + r_k(\gamma) + 1}.
\end{align*}

If $q_k(\gamma) = q_k(\alpha)$ and $r_k(\gamma) < r_k(\alpha)$ then $r_k(\gamma) +1 \leq r_k(\alpha)$. Hence, 
\[ P_{\xi_{q_k(\alpha)}, F(q_k(\alpha)) +r_k(\alpha)} \supseteq P_{\xi_{q_k(\gamma)}, F(q_k(\gamma)) + r_k(\gamma)+1},\] 
and we can apply the same reasoning as above.

Moreover,
\[  Q''_\zeta (\alpha) = Q''_{\xi_{q_k(\alpha)}} (F(q_k(\alpha)) +r_k(\alpha)) \in  P_{\xi_{q_k(\alpha)}, F(q_k(\alpha)) +r_k(\alpha) +1} = P'_\alpha = P''_{\alpha+1}.    \]
This ends the proof of the claim. \end{proof}

We are finally ready to define the $E_\zeta$ witnessing $R(k+1)$. This is 

\[  E_\zeta = \bigcup_{\alpha < \omega_{k+1}} E_{\xi_{q_k(\alpha)}, F(q_k(\alpha)) + \omega_k }. \]

We now define a decomposition $E_\zeta = \bigcup \langle E''_\alpha \mid \alpha \in \{ 0  \} \cup S_k^{k+1} \rangle$ as in the definition.  We let $E''_0 = \emptyset$, and for every $\alpha \in S_k^{k+1}$ we find a set $E''_\alpha$ such that

\begin{enumerate}
    \item $E''_\alpha \sbeq I(Q''_\zeta(\alpha))$,
    \item $E''_\alpha$ is $k$-twistable for $(P''_\alpha, Q''_\zeta \restriction \alpha )$, 
    \item if $\alpha$ is a limit point of $S^{k+1}_k$, then $E''_\alpha \sbeq \bigcup_{\gamma \in \alpha \cap S_k^{k+1}} E''_\gamma$,
\end{enumerate}

    and finally 
    \[  \bigcup_{\alpha\in S_k^{k+1}} E''_\alpha = E_\zeta. \] 

If $q_k(\alpha)$ is a successor we let
\[ E''_\alpha = E_{\xi_{q_k(\alpha)-1}, F(q_k(\alpha)-1) + \omega_k     } \sbeq I(Q''_{\xi_{q_k(\alpha)-1}} (F(q_k(\alpha)-1) + \omega_k )  ) \sbeq I(Q''_\zeta (\alpha)),  \]   
where the last inclusion follows from the fact that 
\begin{align*}
Q''_\zeta (\alpha) 
&= Q''_{\xi_{q_k(\alpha)}} \left( F(q_k(\alpha)) + r_k(\alpha) \right) \\
&\notin P_{\xi_{q_k(\alpha)},\, F(q_k(\alpha)) + r_k(\alpha)} \\
&\supseteq P_{\xi_{q_k(\alpha)},\, F(q_k(\alpha))} \\
&\supseteq P_{\xi_{q_k(\alpha) - 1},\, F(q_k(\alpha) - 1) + \omega_k + 1} \\
&\ni Q''_{\xi_{q_k(\alpha) - 1}} \left( F(q_k(\alpha) - 1) + \omega_k \right),
\end{align*}
which implies $  Q''_\zeta (\alpha) = Q''_{\xi_{q_k(\alpha)}} (F(q_k(\alpha)) + r_k(\alpha))   \not \leq Q''_{\xi_{q_k(\alpha)-1}}  (F(q_k(\alpha)-1) + \omega_k)  $ and hence that  $Q''_{\xi_{q_k(\alpha)-1}}  (F(q_k(\alpha)-1) + \omega_k ) \leq Q''_{\xi_{q_k(\alpha)}} (F(q_k(\alpha)) + r_k(\alpha)) $ since we are dealing with a strong chain.

That $E''_\alpha = E_{\xi_{q_k(\alpha)-1}, F(q_k(\alpha)-1) + \omega_k}$ is $k$-twistable for $(P''_\alpha, Q''_\zeta \restriction \alpha) $ follows from the fact that 
\begin{align*}
P''_\alpha 
&= \bigcup_{\gamma < \alpha} 
    P_{\xi_{q_k(\gamma)},\, F(q_k(\gamma)) + r_k(\gamma) + 1} \\
&= \bigcup_{\beta < \omega_k} 
    P_{\xi_{q_k(\alpha)-1},\, F(q_k(\alpha)-1) + \beta + 1} \\
&= P_{\xi_{q_k(\alpha)-1},\, F(q_k(\alpha)-1) + \omega_k}
\end{align*}
(which is sufficient for the $k=0$ case) and the fact that
\[  Q''_\zeta \restriction (\alpha \setminus q_k(\alpha)-1) = Q''_{\xi_{q_k(\alpha) -1}} \restriction [F(q_k(\alpha)-1), F(q_k(\alpha)-1) + \omega_k),  \]
hence $Q''_\zeta \restriction \alpha$ and $Q''_{\xi_{q_k(\alpha)-1}} \restriction (F(q_k(\alpha)-1) + \omega_k) $ share a final segment and since $E''_\alpha =E_{\xi_{q_k(\alpha)-1}, F(q_k(\alpha)-1) + \omega_k}$ is $k$-twistable for the unbounded pair 
\[( P_{\xi_{q_k(\alpha)-1}, F(q_k(\alpha)-1) + \omega_k}    , Q''_{\xi_{q_k(\alpha)-1}} \restriction (F(q_k(\alpha)-1) + \omega_k)     )\] 
by assumption, then it is for $(P''_\alpha, Q''_\zeta \restriction \alpha)$, thanks to Lemma \ref{lemmafinsegm} in case $k >0$.  

If  $\textrm{cf}(q_k(\alpha)) = \omega_k$ instead, that is, $\alpha = \sup(\alpha \cap S_k^{k+1})$, then we can find an increasing sequence $(\rho_\delta)_{\delta < \omega_k} \sbeq \alpha \cap S_k^{k+1}$ cofinal in $\alpha$ such that $q_k(\rho_\delta)$ is a successor for all $\delta < \omega_k$. Then we can apply $R(k)$ and find a $k$-twistable
\[E''_\alpha \sbeq \bigcup_{\delta < \omega_k} E''_{\rho_\delta} \sbeq  \bigcup_{\gamma \in \alpha \cap S_k^{k+1}} E''_\gamma.  \]
Notice that 
\[ E''_\alpha \sbeq  \bigcup_{\delta < \omega_k} E''_{\rho_\delta}    \sbeq I(Q''_\zeta (\alpha)).          \]
In fact, for all $\delta < \omega_k$,
\[  E''_{\rho_\delta} \subseteq I(Q''_{\xi_{q_k(\rho_\delta) -1}} (F(q_k(\rho_\delta) -1) + \omega_k)) \in P_{\xi_{q_k(\rho_\delta) -1}, F(q_k(\rho_\delta) -1) + \omega_k+1}                              \]
and 
\begin{align*}
P_{\xi_{q_k(\rho_\delta) -1},\, F(q_k(\rho_\delta) -1) + \omega_k + 1} 
&\subseteq P_{q_k(\alpha),\, F(q_k(\alpha))} \\
&\not\ni Q''_{q_k(\alpha)} \left( F(q_k(\alpha)) \right) \\
&= Q''_{q_k(\alpha)} \left( F(q_k(\alpha)) + r_k(\alpha) \right) \\
&= Q''_\zeta (\alpha),
\end{align*}
hence $I(Q''_{\xi_{q_k(\rho_\delta) -1}} (F(q_k(\rho_\delta) -1) )+ \omega_k) \sbeq I(Q''_\zeta (\alpha))$, since we are dealing with a strong chain.

Finally, $E_\zeta \sbeq I(Q'(\zeta))$ because the chain is strong and $E''_\alpha \sbeq I(Q''_\zeta(\alpha)) \sbeq I(Q'(\zeta))$ for all $\alpha \in S_k^{k+1}$.
\end{proof}

\begin{lemma}
\label{continuity}
For all $k\geq 0$, $R(k)$ holds.

\end{lemma}

\begin{proof}

Lemma \ref{0cont} easily implies the base case $R(0)$ since we are dealing with a strong chain and hence we find a $0$-twistable $E_\zeta \sbeq \bigcup_{i < \omega} E_{\xi_i} \sbeq \bigcup_{i < \omega} I(Q'(\xi_i)) \sbeq I(Q'(\zeta))$. Lemma $\ref{contind}$ gives us the induction step. We conclude by induction.
\end{proof}

Finally, by continuity, we conclude the existence of a twistable set for any unbounded pair.

\begin{lemma}

\label{exitwist}

For every $k \geq 0$ and for every $(k{+}1)$-spine filtration 
\[
\left( \left\langle P_\alpha \mid \alpha < \omega_{k+1} \right\rangle,\, Q' \right)
\]
of a strong $(k{+}1)$-unbounded pair $(P, Q)$, there exists a $(k{+}1)$-twistable set 
\[
E \subseteq \bigcup_{f \in Q} I(f).
\]
\end{lemma}    

\begin{proof}

We prove the result by induction. Let $(\lang P_\alpha \mid \alpha < \omega_{k+1} \rang , Q')$ be our spine filtration. If $\alpha \in S_k^{k+1}$ and $\alpha > \sup(\alpha \cap S_k^{k+1}) $ then $(P_\alpha, Q' \restriction \alpha)$ is a strong $k$-unbounded pair. If $k=0$ then use Lemma \ref{0exist} to find a $0$-twistable set $E_\alpha \sbeq I(Q'(\alpha))$. If $k>1$ then, by the induction assumption, there exists a $k$-twistable $E \sbeq \bigcup_{f \in Q' \restriction \alpha } I(f) \sbeq I(Q'(\alpha)) $, where the last inclusion follows from the fact that $Q'$ is a strong chain. In both cases we find $E_\alpha \sbeq \bigcup_{\gamma \in \alpha \cap S_k^{k+1}} E_\gamma $ for the case of $\alpha \in S_k^{k+1}$ with $\alpha = \sup(\alpha \cap S_k^{k+1}) $ as follows. 

First, we find $(\rho_\delta)_{\delta < \omega_k} \sbeq \alpha \cap S_k^{k+1}$ such that $\rho_\delta > \sup(\rho_\delta \cap S_k^{k+1})$ for all $\delta < \omega_k$. Then we apply continuity, that is, $R(k)$ of Lemma \ref{continuity} to find $E_\alpha$.   

Finally, we let $E_0 = \emptyset$ and obtain that $E = \bigcup \langle E_\alpha \mid \alpha \in \{0\} \cup S_k^{k+1} \rangle $ is $(k+1)$-twistable and  $E \sbeq \bigcup_{\alpha \in S_k^{k+1}} I(Q'(\alpha)) \sbeq \bigcup_{f \in Q} I(f)$.
\end{proof}

\section{Nontriviality on unbounded chains}
\label{sect:nontriviality}

Let $\langle(P_\alpha, Q_\alpha) \mid \alpha \in S_n^{n+1} \rangle$ be a sequence of $n$-unbounded pairs coming from a tree filtration of $ {}^{\omega_n} \omega$. In this section, we prove by induction the existence of a nontrivial coherent $n$-family on $P_\alpha$, whose nontriviality is witnessed by the values taken on an $n$-twistable set by the functions whose multi-index is in the chain $Q_\alpha$.
The base case of the induction is a $\mathsf{ZFC}$ fact.  

\begin{lemma}
\label{lemmaT1}
For any $1$-twistable set $E=\bigcup_{\gamma\in S^1_0\cup\{0\}} E_\gamma$ for a spine filtration $(\lang P_\alpha \rang_{\alpha < \omega_1}, Q')$ of a strong $\aleph_1$-unbounded pair $(P,Q)$ in $({^\lambda}\omega,\leq)$, there exists a coherent family
$$\Phi=\langle\varphi_f:I(f)\to\mathbb{Z}\mid f\in P\rangle$$ with the property that $$(\Phi \restriction Q') \restriction \restriction E = \langle\varphi_f\restriction I(f)\cap E\mid f\in Q'\rangle$$ is nontrivial.
\end{lemma}

\begin{proof}
Here the broad idea is that the decomposition $E=\bigcup_{\gamma\in S^1_0} E_\gamma$ sufficiently resembles the decomposition $\omega_1\times\omega=\bigcup_{\gamma\in\omega_1}\{\gamma\}\times\omega$ to support constructions along the lines of Lemma \ref{lem:1} (which trace, in turn, to those of Example \ref{ex:fintoone}). 
Before more precisely describing such a construction, let us record a few observations. First, it is among our assumptions that $P$ possesses a cofinal directed set of the form $P'=\langle p_\alpha \mid \alpha<\omega_1\rangle$; thus to prove the lemma, it will suffice by the remarks concluding Section \ref{sect:ntcs+Ak} to construct a coherent family $\Phi\restriction P'\cup Q'=\left<\varphi_f:I(f)\to \mathbb{Z}\mid f\in P'\cup Q'\right>$ with the property that $(\Phi\restriction Q')\restriction\restriction E$ is nontrivial.
%To that end, enumerate $Q$ as $\langle q_\alpha \mid \alpha<\omega_1\rangle$.
Second, note that $E_\alpha\subseteq I(Q'(\beta))$ for all $\alpha\leq\beta \in S_0^1$, thus if $\gamma$ is a successor of some $\beta$ in the natural enumeration of $S_0^1$ then $\bigcup_{\alpha\in S^1_0\cap\gamma} E_\alpha$ has finite intersection with $E_\gamma$; write $S\subset S^1_0$ for the collection of such $\gamma$.
Relatedly and third, note that for every $f\in P$ there exists a least $\beta\in S^1_0$ such that $\gamma\in S^1_0\backslash\beta$ implies that $E_\gamma\cap I(f)$ is finite.
Write $A$ for the set of accumulation points of $S_0^1$ below $\omega_1$ and let $C$ denote the closed unbounded subset
$$\{\beta\in A\mid E_\gamma\cap I(p_\alpha)\text{ and }E_\gamma\cap I(Q'(\alpha))\text{ are finite for all }\alpha<\beta\text{ and }\gamma\in S^1_0\backslash\beta\}$$
of $\omega_1$; note that we may without loss of generality assume the collection of functions of the form $p_\alpha$ or $Q(\alpha)$ for some $\alpha<\beta$ to be $\leq^*$-directed for each $\beta\in C$ as well.
Fix also a filtration of $E = \bigcup_{i<\omega} E^i$ by sets $E^i$ whose intersection with $E_\gamma$ is finite for every $(i,\gamma)\in \omega\times S_0^1$, and choose $e^i_\gamma\in E_\gamma\backslash E^i$ so that $i\neq j$ implies $e^i_\gamma\neq e^j_\gamma$ for all $i,j\in\omega$ and $\gamma\in S$.
We will construct $\Phi\restriction P'\cup Q'$ by recursion on the elements $\xi$ of $C$, maintaining as we go the hypotheses that for all $\eta\in C\cap\xi$ and $\alpha<\eta\leq\delta<\xi$,
$$\{\gamma\in S^1_0\backslash\eta\mid I(p_\alpha)\cap\mathrm{supp}(\varphi_{Q'(\delta)})\cap E_\gamma\neq\varnothing\}$$
and
$$\{\gamma\in S^1_0\backslash\eta\mid I(Q'(\alpha))\cap\mathrm{supp}(\varphi_{Q'(\delta)})\cap E_\gamma\neq\varnothing\}$$
are finite, as is
\begin{equation}
\label{cond:IH3}
\{\gamma\in S^1_0\mid E^i\cap\mathrm{supp}(\varphi_{Q'(\delta)})\cap E_\gamma\neq\varnothing\}\end{equation}
for every $i\in\omega$ and $\delta<\xi$.
The third of these hypotheses is, in essence, a finite-to-one-ness condition which, alongside the condition that
\begin{equation}
\label{cond:IH4}
\mathrm{supp}(\varphi_{Q'(\delta)})\cap E_\gamma\neq\varnothing\text{ for all }\gamma\in S\cap\delta\text{ and }\delta<\xi
\end{equation}
will ensure, at the construction's conclusion, the nontriviality of $\Phi\restriction Q'$, while the first two hypotheses will facilitate our construction's continuation.
Write $\mathrm{RH}(\xi)$ for the conjunction of these conditions, and note that it's trivially conserved at limit steps; thus it will suffice to establish $\mathrm{RH}(\xi)$ at base and successor steps $\xi\in C$.

For the base case $\xi = \min C$, write $\Phi_\xi$ for the family
$$\left<\varphi_f: I(f)\to\mathbb{Z}\mid f = p_\alpha\text{ or }f = Q'(\alpha)\text{ for some }\alpha<\xi\right>$$ and define the functions of $\Phi_\xi\restriction\restriction (\lambda\backslash\xi\times\omega)$ all to constantly equal $0$, and, using an ordertype-$\omega$ enumeration of $S\cap\xi$, define $\Phi_\xi\restriction\restriction \xi\times\omega$ to satisfy the only non-vacuous conditions of $\mathrm{RH}(\xi)$, namely items (\ref{cond:IH3}) and (\ref{cond:IH4}) above; note that it's straightforward to additionally arrange that the supports of all functions in $\Phi_\xi$ fall within $E$.

Suppose next that $\xi$ is the successor in $C$ of some $\eta$ and that $\mathrm{RH}(\eta)$ holds.
By Goblot's theorem, there exists a family $\varphi:\eta\times\omega\to\mathbb{Z}$ trivializing $\Phi_\eta\restriction\restriction\eta\times\omega$; fix also a sequence $C_\eta=\langle\eta_i\mid\omega\rangle$ which is increasing and cofinal in $\eta$.
To define a $\varphi_{Q'(\eta)}$ satisfying the $\eta=\delta$ instance of (\ref{cond:IH3}) and (\ref{cond:IH4}) above, we revise $\varphi$ along $C_\eta$ as follows: for $x\in I(Q'(\eta))$, we let
$\varphi_{Q'(\eta)}(x)= 
\varphi(x)$ \emph{unless $x\in E_\gamma$ for some  $\gamma\in S\cap [\eta_j,\eta_{j+1})$ with $\mathrm{supp}(\varphi)\cap E_\gamma\cap E_j\neq\varnothing$}.
For such $x$ and $E_\gamma$, let
\[
\varphi_{Q'(\eta)}(x)=  \begin{cases} 
1 & \text{if } x= e^j_\gamma \\
0 & \textrm{otherwise} 
\end{cases}
\]
Next, fix bijections $b:\omega\to\eta$ and $a:\omega\to S^1_0\cap [\eta,\xi)$ and for each $\beta\in S^1_0\cap [\eta,\xi)$ define $\varphi_{Q'(\beta)}(x)$ by
\[
\varphi_{Q'(\beta)}(x)=  \begin{cases} 
\varphi_{Q'(\eta)}(x) & \text{if } x\in I(Q'(\eta)) \\
1 & \text{if }x=e^i_\delta\text{ for some }\delta\in S\cap [\eta,\beta)\text{ with }i\text{ the minimum of } \\
& \hspace{.15 cm}\{j\in\omega\mid e_\delta^j\not\in I(p_{b(\alpha)})\cup I(Q'(b(\alpha)))\text{ for any }\alpha\in a^{-1}(\delta)\}\\
0 & \textrm{otherwise} 
\end{cases}
\]
Again by Goblot's theorem, there exists a trivialization $\psi:\xi\times\omega\to\mathbb{Z}$ of the (suitable restrictions of the) functions so far defined; one then defines the restriction to $\xi\times\omega$ of all yet-undefined $\varphi_{p_\alpha}$ and $\varphi_{Q'(\alpha)}$ with $\alpha<\xi$ as $\psi\restriction I(p_\alpha)$ and $\psi\restriction I(Q'(\alpha))$, respectively, and lets the remainder of these functions constantly equal $0$.
This completes the construction.
\end{proof}

The easy verification of the following fact, which will be invoked several times below, is left to the reader.
\begin{lemma}
\label{lemmarestriction}
For every pair of directed sets $P \subseteq P' \subseteq {}^{\omega_n} \omega$ and domains $E \subseteq E' \subseteq \omega_n \times \omega$, if $\Psi = \{ \psi_{\vec{f}} : E' \to \mathbb{Z} \mid \vec{f} \in P'^{(k)} \} $ trivializes $\Phi$, then $(\Psi \restriction P^{(k)}) \restriction \restriction E$ trivializes $(\Phi \restriction P^{(k)}) \restriction \restriction E$.
\end{lemma}

%A few points to keep in mind here:
%\begin{itemize}
%\item To my mind, there's no strong reason to switch the codomain to more general (abelian) groups at this point (more groups just means more counterexamples to additivity, but one already is bad enough -- and we can always just note that things work in greater generality in a brief remark); a main reason for doing so might be to facilitate the coding. But how we approach this is up to you; the main point is that whatever choice we do make here should be briefly explained to the reader.
%\item We haven't yet had occasion to introduce your $\delta$ notation for differentials, so it should be explained if you plan to make use of it. Since $\delta$ in these contexts also tends to denote an ordinal, $\partial$ might be a better symbol for the differential.
%\item When introducing the diamond hypotheses, you might say a bit about how weak diamond suffices (which seems to me a much more significant greater generality than varying the codomain groups here).]
%\end{itemize}

We turn now to the induction step of our argument; this is based on some guessing principles. These allow us to construct nontrivial coherent families by trasfinite recursion. More precisely, they allow us to negate all putative trivializations by choosing between a pair of extensions one that is not trivialized by any extension of the family coded by the set $A_\alpha$ in the $\lozenge(S^{k+1}_k)$-sequence $\langle A_\alpha \mid \alpha \in S^{k+1}_k \rangle$. 

For $k \geq 1$, let $T(k)$ denote the claim that for every strong $k$-unbounded pair $(P,Q)$, every one of its $k$-spine filtrations $(\langle P_\alpha \mid \alpha < \omega_k\ \rangle, Q')$ and every $k$-twistable set $E$ for $(\lang P_\alpha \mid \alpha < \omega_k \rangle, Q') $ there exists a $k$-coherent family $\Phi$ such that $(\Phi \restriction Q'^{(k)}) \restriction \restriction E$ is nontrivial. Note that Lemma \ref{lemmaT1} above is $T(1)$.

\begin{lemma}
\label{lemma:induction}
For every $k \geq 1$, if $T(k)$ and $\lozenge(S_k^{k+1})$ hold, then $T(k+1)$ holds.

\end{lemma}

\begin{proof}

Let $E= \bigcup \mathcal{E}$ with $\mathcal{E} = \langle E_\alpha \mid \alpha \in \lbrace 0 \rbrace \cup S_k^{k+1} \rangle $ be our $(k+1)$-twistable set. For $\alpha \in S_k^{k+1}$, let $\mathcal{P}^\alpha$ be the set of all families of the form $\lbrace \psi_{\vec{f}}: \bigcup \mathcal{E} \restriction (\alpha +1)  \cap I(f_0) \to \Z \mid \vec{f} \in (Q' \restriction \alpha)^{(k)} \rbrace $. Similarly, let $\mathcal{P}^{\omega_{k+1}}$ be the set of all families of the form $ \lbrace  \psi_{\vec{f}}: E  \cap I(f_0) \to \Z   \mid  \vec{f} \in Q'^{(k)}    \rbrace  $

We fix a coding function $G$ with domain ${}^{\leq \omega_{k+1}} 2$ such that

\begin{itemize}
    \item for all $\alpha \in S_k^{k+1} \cup \{\omega_{k+1}\}$, $G \restriction {^{\alpha}}2$ is a surjection 
    from ${^{\alpha}}2$ to $\mathcal{P}^\alpha$.
    \item for all $\alpha < \beta$, both in $S_k^{k+1} \cup \{\omega_{k+1}\}$ and all $\sigma \in {^{\alpha}}2$ 
    and $\sigma' \in {^{\beta}}2$, if $\sigma'$ end-extends $\sigma$, then the family  
    $G(\sigma')$ extends the family $ G(\sigma)$.

\end{itemize}

We show that such a $G$ exists by recursion on the increasing enumeration of $S_k^{k+1} \cup \{ \omega_{k+1} \}$. 

Let $\alpha \in S_k^{k+1}$ and assume that the conditions are satisfied for all lower members of the stationary set.
If $\alpha > \sup(\alpha \cap S_k^{k+1})$, then $\alpha = \sup (\alpha \cap S_k^{k+1}) + \omega_k$. Since $\vert \bigcup \mathcal{E} \restriction (\alpha +1) \vert \leq \aleph_k$ by Lemma \ref{sizetwist},  we can use the last $\omega_k$ values of $\sigma \in {}^\alpha 2$ to code the values of all possible extensions of $\bigcup_{\gamma \in \alpha \cap S_k^{k+1}} G(\sigma \restriction \gamma)$ on $\bigcup \mathcal{E} \restriction (\alpha+1)$ and for multi-indices in $(Q' \restriction \alpha)^{(k)}$, thus preserving the two conditions above.

If $\alpha = \sup(\alpha \cap S_k^{k+1})$, then we use the fact that $E_\alpha \subseteq \bigcup \mathcal{E} \restriction\alpha$ and let $G(\sigma) = \bigcup_{\gamma \in \alpha \cap S_k^{k+1}} G(\sigma \restriction \gamma)$.

Finally, for $\sigma \in {}^{\omega_{k+1}} 2$ we let $G(\sigma) = \bigcup_{\alpha \in  S_k^{k+1}} G(\sigma \restriction \alpha)$.

Now we define a sequence of $(k+1)$-coherent (and trivial) families $\langle \Phi^\alpha \mid \alpha < \omega_{k+1}\rangle$ such that $\Phi = \bigcup_{\alpha < \omega_{k+1}} \Phi^\alpha$ is nontrivial.

At limit $\lambda$, let $\Phi^\lambda = \bigcup_{\alpha< \lambda} \Phi^\alpha$. At successor step, we know by Goblot's theorem that $\Phi^\alpha$ is trivial, so we let $ \partial(\Theta^\alpha) =  \Phi^\alpha$, then we just consider the extension $\widetilde{\Theta}^\alpha$ defined as 
\[
\widetilde{\Theta}_{\vec{f}}^\alpha = 
\begin{cases}
\Theta^\alpha_{\vec{f}} & \text{if } \vec{f} \in P_\alpha^{(k)} \\
0 & \text{if } \vec{f} \in P_{\alpha+1}^{(k)} \setminus P_\alpha^{(k)}
\end{cases}
\]

Now, if $\mathrm{cf}(\alpha) < \omega_k$, then we simply let $\Phi^{\alpha+1} = \partial(\widetilde{\Theta}^\alpha)$.

If $\mathrm{cf}(\alpha) = \omega_k $ instead, then $(P_\alpha, Q' \restriction \alpha) $ is a strong $k$-unbounded pair, so by the induction assumption $T(k)$, there exists a $k$-coherent family $\Xi^\alpha$ on $P_\alpha$ such that $\Xi^\alpha \restriction (Q' \restriction \alpha)^{(k)} \restriction \restriction E_\alpha$ is nontrivial. Similarly, we let
\[
\widetilde{\Xi}_{\vec{f}}^\alpha = 
\begin{cases}
\Xi^\alpha_{\vec{f}} & \text{if } \vec{f} \in P_\alpha^{(k)} \\
0 & \text{if } \vec{f} \in P_{\alpha+1}^{(k)} \setminus P_\alpha^{(k)}
\end{cases}
\]
 
Then we let $\Phi^{\alpha, 0} = \partial(\widetilde{\Theta}^\alpha     )$, and $\Phi^{\alpha, 1} = \partial( \widetilde{\Theta}^\alpha +\widetilde{\Xi}^\alpha)$. We need to decide which to use as our $\Phi^{\alpha+1}$. Let $\langle A_\alpha \mid \alpha \in S_k^{k+1} \rangle$ be a $\lozenge(S_k^{k+1})$-sequence. Consider $A_\alpha \in {}^\alpha 2$. If $\Psi_{A_\alpha} = G(A_\alpha)$ does not trivialize the family  $(\Phi^\alpha \restriction (Q' \restriction \alpha)^{(k+1)}) \restriction \restriction E_\alpha$, then put $\Phi^{\alpha +1} = \Phi^{\alpha,  \epsilon}$ for any $\epsilon \in \lB 0,1 \rB $. If it does, consider the following claim:

\begin{claim}
For every trivialization $ \Upsilon$ of $(\Phi^\alpha \restriction (Q' \restriction \alpha)^{(k+1)}) \restriction \restriction E_\alpha$ there exists $\epsilon \in \lB 0,1 \rB $ such that $\Upsilon$ does not extend to a trivialization of the family  $(\Phi^{\alpha, \epsilon} \restriction (Q' \restriction (\alpha+1))^{(k+1)}) \restriction \restriction E_\alpha$ 
\end{claim}

\begin{proof}[Proof of Claim]
Suppose toward a contradiction that there exist $\Upsilon^0$ and $\Upsilon^1$ that trivialize $(\Phi^{\alpha, 0} \restriction (Q' \restriction (\alpha+1))^{(k+1)}) \restriction \restriction E_\alpha$ and  $(\Phi^{\alpha, 1} \restriction (Q' \restriction (\alpha+1))^{(k+1)}) \restriction \restriction E_\alpha$ respectively, and $\Upsilon^0 \restriction (Q'\restriction \alpha)^{(k)} = \Upsilon = \Upsilon^1 \restriction (Q'\restriction \alpha)^{(k)} $. 

We define $d_f(\Upsilon)$ by $( d_f (\Upsilon))_{\vec{h}}  = \Upsilon_{\vec{h}, f }  $ for $\vec{h} = (h_0, ... , h_{k-2})$. If $f = Q'(\alpha)$ and $\vec{g} = (g_0, ... , g_{k-1})$ ranges over $(Q' \restriction \alpha)^{(k)} $ we have (omitting restrictions to $I_{g_0} \cap E_\alpha$ for readability):
\begin{align*}
\partial ( d_f ( \Upsilon^0)  )_{\vec{g}} 
&= \sum_{0 \leq i \leq k-1} (-1)^i (d_f (\Upsilon^0))_{\vec{g}^i} \\
&= \sum_{0 \leq i \leq k-1} (-1)^i \Upsilon^0_{\vec{g}^i, f} \\
&=^\ast (-1)^{k-1} \Upsilon^0_{\vec{g}} + \Phi^{\alpha, 0}_{\vec{g}, f} \\
&= (-1)^{k-1} \Upsilon^0_{\vec{g}} + (-1)^k \Theta^\alpha_{\vec{g}}.
\end{align*} 

By the same reasoning 
\[   \partial  ( d_f ( \Upsilon^1)  )_{\vec{g}}  =^\ast  (-1)^{k-1} \Upsilon^1_{\vec{g}} + (-1)^k \cdot ( \Theta^\alpha_{\vec{g}} + \Xi^{\alpha}_{\vec{g}} ),    \]
so that, since $\Upsilon^0 $ and $\Upsilon^1$ agree on $(Q' \restriction \alpha)^{(k)}$,
\[ \partial ( d_f ( \Upsilon^1 - \Upsilon^0 )  ) = (-1)^k (\Xi^\alpha \restriction (Q' \restriction \alpha)^{(k)} )\restriction \restriction E_\alpha ,  \]
thus contradicting the fact that $(\Xi^\alpha \restriction (Q' \restriction \alpha)^{(k)}) \restriction \restriction E_\alpha$ is nontrivial.

It should be noticed that the case $k = 1$ is computationally identical but conceptually different insofar as we need to stipulate that a family indexed by $0$-tuples is a singleton and that a function defined on $I(f) \cap E_\alpha$ can be regarded as a trivialization of a $1$-family below $f$ as it can be extended to the entire grid $\omega_n \times \omega$.
\end{proof}

Then we let $\Phi^{\alpha+1} = \Phi^{\alpha, \epsilon}$ for the $\epsilon \in \{0, 1\}$ such that $\Psi_{A_\alpha}$ does not extend to a trivialization of $(\Phi^{\alpha, \epsilon} \restriction (Q' \restriction (\alpha+1))^{(k+1)}) \restriction \restriction E_\alpha$. 

Finally, we let $\Phi = \bigcup_{\alpha < \omega_{k+1}} 
 \Phi^\alpha$. 

We check that $(\Phi \restriction Q'^{(k+1)}) \restriction \restriction E$ is nontrivial. Let $\Psi$ be any putative trivialization, and let $A \in {}^{\omega_{k+1}} 2 $ be such that $G(A)= \Psi$. Let $\alpha \in S_k^{k+1} $ be such that $A_\alpha = A\cap \alpha$. So, $G(A_\alpha) = \Psi_{A_\alpha} =  (\Psi \restriction (Q' \restriction \alpha)^{(k)})  \restriction \restriction \bigcup \mathcal{E} \restriction (\alpha+1)$. By Lemma \ref{lemmarestriction}, $\Psi_{A_\alpha}$ should trivialize $(\Phi^\alpha \restriction (Q' \restriction \alpha)^{(k+1)}) \restriction \restriction E_\alpha$. Therefore, $(\Phi^{\alpha+1} \restriction (Q' \restriction (\alpha+1))^{(k+1)} )\restriction \restriction E_\alpha $ is not trivialized by any extension of $\Psi_{A_\alpha}$ such as $(\Psi \restriction (Q' \restriction (\alpha+1)^{(k)})) \restriction \restriction \bigcup \mathcal{E} \restriction (\alpha+1) $. Since $Q'\restriction (\alpha+1) \subseteq Q'$ and $\bigcup \mathcal{E} \restriction (\alpha+1) \subseteq E$, again by Lemma \ref{lemmarestriction},  $\Psi$ is not a trivialization of $(\Phi \restriction Q'^{(k+1)}) \restriction \restriction E$.
\end{proof}

\begin{theorem}

\label{nontrivtwist}

Let $k \geq 1$. If $\bigwedge_{i < k} \Diamond(S_i^{i+1})$ holds, then the statement $T(k)$ holds. 

\end{theorem}

\begin{proof}
The result easily follows from Lemma \ref{lemmaT1} and Lemma \ref{lemma:induction}. 
\end{proof}

\section{Theorem B}
\label{sect:thmB}

The following implies Theorem B and will therefore conclude our argument.
To simplify its statement, we invoke stronger hypotheses here than are really necessary.
This is true in two senses: first, by Lemma \ref{lemmaT1}, $\Diamond(S_0^1)$ is unneeded, and the statement of Theorem B reflects this fact. Second, even for $i>0$, the assumptions $\Diamond(S_i^{i+1})$ are a little stronger than we need. In fact, the weak diamond principles $\mathrm{w}\lozenge(S_i^{i+1})$ suffice. To see this, one simply needs to run the argument of \cite[Claim 4.11]{Cas24} with domains $\bigcup \mathcal{E} \restriction(\alpha+1)$, and use the fact that these are monotone increasing in $\alpha$. 

\begin{theorem}
For any $n>0$, if $2^{\aleph_i}=\aleph_{i+1}$ and $\Diamond(S_i^{i+1})$ hold for all $i\leq n$ then $\mathrm{lim}^{n+1} \A_{\aleph_n} \neq 0$.
\end{theorem}

\begin{proof}
First note our premises subsume those of Lemma \ref{lemmatree}, thus below, we can and will employ the conclusions and notations of the latter. We have then a continuous filtration ${}^{\omega_n} \omega = \bigcup_{\alpha < \omega_{n+1}} P_\alpha $ and strong $n$-unbounded pairs $(P_\alpha, Q_\alpha)$ for all $\alpha  \in S_n^{n+1}$. From Lemma \ref{exitwist} we deduce the existence of an $n$-twistable set $E^\alpha \subseteq \bigcup_{f \in Q_\alpha} I(f) \subseteq I(f_\alpha)$ for $(P_\alpha, Q_\alpha)$. From Theorem \ref{nontrivtwist}  and $\bigwedge_{i < n} \Diamond(S_i^{i+1})$ we derive the existence of an $n$-coherent nontrivial family $\Xi^\alpha$ on $P_\alpha$ such that $(\Xi^\alpha \restriction {Q'}_\alpha^{(n)}) \restriction \restriction E^\alpha$ is nontrivial. By Lemma \ref{lemmarestriction}, we conclude that $\Xi^\alpha \restriction Q_\alpha^{(n)}$ itself is nontrivial.

Now we want to code the restriction of any putative trivialization $\Psi$ to $\mathrm{im}(H)$ in an efficient way. 

For $\alpha \in S_n^{n+1}$, let  $\mathcal{R}^\alpha$ be the set of all $n$-families indexed over $(H[T' \restriction\alpha])^{(n)}$. Similarly, let $\mathcal{R}^{\omega_{n+1}}$ be the set of all families indexed over $(\mathrm{im}(H))^{(n)}$.

We fix a coding function $G$ with domain ${}^{\leq \omega_{n+1}} 2$ such that

\begin{itemize}
    \item for all $\alpha \in S_n^{n+1} \cup \{\omega_{n+1}\}$, $G \restriction {^{\alpha}}2$ is a surjection 
    from ${^{\alpha}}2$ to $\mathcal{R}^\alpha$.
    \item for all $\alpha < \beta$, both in $S_n^{n+1} \cup \{\omega_{n+1}\}$ and all $\sigma \in {^{\alpha}}2$ 
    and $\sigma' \in {^{\beta}}2$, if $\sigma'$ end-extends $\sigma$, then the family  
    $G(\sigma')$ extends the family $ G(\sigma)$.

\end{itemize}

We show that such a $G$ exists by recursion on the increasing enumeration of $S_n^{n+1} \cup \{ \omega_{n+1} \}$. Let $\alpha \in S_n^{n+1}$ and assume the coding has already been defined for all ordinals in $\alpha \cap S_n^{n+1}$. 

If $\alpha >\sup(\alpha  \cap S_n^{n+1})$ then $\alpha = \sup (\alpha \cap S_n^{n+1}) + \omega_n$, and we can use the last $\omega_n$ values of $\sigma \in {}^\alpha 2$ to code the values of all possible extensions of $\bigcup_{\gamma \in \alpha \cap S_n^{n+1}} G(\sigma \restriction \gamma)$ (here we are using that $\vert \mathrm{lev}_\alpha (T') \vert \leq \aleph_n$ for all $\alpha$.

If $\alpha = \sup(\alpha  \cap S_n^{n+1})$, we simply let $G(\sigma) = \bigcup_{\gamma \in \alpha \cap S_n^{n+1}} G(\sigma \restriction \gamma)$.

Finally, if $\sigma \in {}^{\omega_{n+1}} 2$, we let $G(\sigma) = \bigcup_{\alpha \in S_n^{n+1}} G(\sigma \restriction \alpha)$.

Now we construct a sequence $\langle \Phi^\alpha \mid \alpha < \omega_{n+1} \rangle $ of $(n+1)$-coherent families extending one another, where $\Phi^\alpha$ is indexed over $P_\alpha$, and $\Phi = \bigcup_{\alpha < \omega_{n+1}} \Phi^\alpha$ is nontrivial.

At limits let $\Phi^\lambda = \bigcup_{\alpha < \lambda} \Phi^\alpha$. At successors use Goblot's theorem and find $ \partial(\Theta^\alpha) =  \Phi^\alpha$. Again we consider the extension $\widetilde{\Theta}^\alpha$ defined as 

\[
\widetilde{\Theta}_{\vec{f}}^\alpha = 
\begin{cases}
\Theta^\alpha_{\vec{f}} & \text{if } \vec{f} \in P_\alpha^{(n)} \\
0 & \text{if } \vec{f} \in P_{\alpha+1}^{(n)} \setminus P_\alpha^{(n)}
\end{cases}
\]

Now, if $\mathrm{cf}(\alpha) < \omega_n$, then we simply let $\Phi^{\alpha+1} = \partial(\widetilde{\Theta}^\alpha)$.

If $\mathrm{cf}(\alpha) = \omega_n $ instead,  take the coherent nontrivial family $\Xi^\alpha$ $P_\alpha$ defined above such that $\Xi^\alpha \restriction Q_\alpha^{(n)} $ is nontrivial. Then let

\[
\widetilde{\Xi}_{\vec{f}}^\alpha = 
\begin{cases}
\Xi^\alpha_{\vec{f}} & \text{if } \vec{f} \in P_\alpha^{(n)} \\
0 & \text{if } \vec{f} \in P_{\alpha+1}^{(n)} \setminus P_\alpha^{(n)}
\end{cases}
\]

Finally, let $\Phi^{\alpha, 0} = \partial(\widetilde{\Theta}^\alpha     )$, and $\Phi^{\alpha, 1} = \partial( \widetilde{\Theta}^\alpha +\widetilde{\Xi}^\alpha)$. We need to decide which to use as our $\Phi^{\alpha+1}$. Let $\langle A_\alpha \mid \alpha \in S_n^{n+1} \rangle$ be a $\lozenge(S_n^{n+1})$-sequence. Consider $A_\alpha \in {}^\alpha2$. Recall that it codes a family $\Psi_{A_\alpha} = G(A_\alpha)$ indexed over $(H [T'\restriction \alpha])^{(n)}$, so in particular it codes some $\Psi_{A_\alpha} \restriction Q_\alpha^{(n)}$. If $\Psi_{A_\alpha} \restriction Q_\alpha^{(n)}$ does \textit{not} trivialize the family  $\Phi^\alpha \restriction Q_\alpha^{(n)}$, then let $\Phi^{\alpha +1} = \Phi^{\alpha,  \epsilon}$ for any $\epsilon \in \lB 0,1 \rB $. If it does, consider the following claim:

\begin{claim}
For every trivialization $ \Upsilon$ of $\Phi^\alpha \restriction Q_\alpha^{(n+1)} $ there exists $\epsilon \in \lB 0,1 \rB $ such that $\Upsilon$ does not extend to a trivialization of the family $\Phi^{\alpha, \epsilon} \restriction  (Q_\alpha \cup \lB f_\alpha \rB)^{(n+1)} $.
\end{claim}

\begin{proof}[Proof of Claim]

Suppose toward a contradiction that there exist $\Upsilon^0$ and $\Upsilon^1$ that trivialize $\Phi^{\alpha, 0} \restriction (Q_\alpha \cup \lB f_\alpha \rB)^{(n+1)} $ and  $\Phi^{\alpha, 1} \restriction (Q_\alpha \cup \lB f_\alpha \rB)^{(n+1)}$ respectively, and $\Upsilon^0 \restriction Q_\alpha^{(n)} = \Upsilon = \Upsilon^1 \restriction Q_\alpha^{(n)} $. 

We define $d_f(\Upsilon)$ by $( d_f (\Upsilon))_{\vec{h}}  = \Upsilon_{\vec{h}, f }  $ where $\vec{h}= (h_0, ..., h_{n-2}) $. If $f = f_\alpha$ and $\vec{g} = (g_0, ... , g_{n-1})$ ranges over $Q_\alpha^{(n)} $ we have (omitting restrictions to $I_{g_0} $ for readability):
\begin{align*}
\partial ( d_f ( \Upsilon^0)  )_{\vec{g}} 
&= \sum_{0 \leq i \leq n-1} (-1)^i (d_f (\Upsilon^0))_{\vec{g}^i} \\
&= \sum_{0 \leq i \leq n-1} (-1)^i \Upsilon^0_{\vec{g}^i, f} \\
&=^\ast  (-1)^{n-1} \Upsilon^0_{\vec{g}} + \Phi^{\alpha, 0}_{\vec{g}, f} \\
&= (-1)^{n-1} \Upsilon^0_{\vec{g}} + (-1)^n \Theta^\alpha_{\vec{g}}.
\end{align*}

By the same reasoning 
\[   \partial  ( d_f ( \Upsilon^1)  )_{\vec{g}}  =  (-1)^{n-1} \Upsilon^1_{\vec{g}} + (-1)^n \cdot ( \Theta^\alpha_{\vec{g}} + \Xi^{\alpha}_{\vec{g}} ),    \]

so that, since $\Upsilon^0 $ and $\Upsilon^1$ agree on $Q_\alpha^{(k)} $,
\[ \partial ( d_f ( \Upsilon^1 - \Upsilon^0 )  ) = (-1)^n \Xi^\alpha \restriction Q_\alpha^{(n)},  \]
contradicting the fact that $\Xi^\alpha \restriction Q_\alpha^{(n)}$ is nontrivial. 

Once again, the case $n = 1$ is computationally identical but conceptually different insofar as we need to stipulate that a family indexed by $0$-tuples is a singleton and that a function defined on $I(f)$ can be regarded as a trivialization of a $1$-family below $f$ as it can be extended to the entire grid $\omega_n \times \omega$.
\end{proof}

Then we let $\Phi^{\alpha+1} = \Phi^{\alpha, \epsilon}$ for the $\epsilon \in \{0,1 \}$ such that $\Psi_{A_\alpha} \restriction Q_\alpha^{(n)}$ does not extend to a trivialization of $\Phi^{\alpha, \epsilon} \restriction (Q_\alpha \cup \{ f_\alpha \})^{(n+1)}$.
Finally, let $\Phi = \bigcup_{\alpha < \omega_{n+1}} 
 \Phi^\alpha$. 

We check that $\Phi$  is nontrivial. Let $\Psi$ be any putative trivialization, and let $A \in {}^{\omega_{n+1}}2$ be such that  $G(A) = \Psi \restriction (\mathrm{im}(H))^{(n)}$. Let $ \alpha \in S_n^{n+1} $ be such that $A_\alpha = A\cap \alpha$. Then $\Psi_{A_\alpha} \restriction Q_\alpha^{(n)} = \Psi \restriction Q_\alpha^{(n)}$. By Lemma \ref{lemmarestriction}, $\Psi \restriction Q_\alpha^{(n)}  $ should trivialize $\Phi^\alpha \restriction Q_\alpha^{(n+1)}$. Therefore, $\Phi^{\alpha+1} \restriction (Q_\alpha \cup \{ f_\alpha \})^{(n+1)}$ is not trivialized by any extension of $\Psi_{A_\alpha} \restriction Q_\alpha^{(n)}$ such as $\Psi \restriction (Q_\alpha \cup \{ f_\alpha \})^{(n)}$. Again by Lemma \ref{lemmarestriction}, $\Psi$ is not a trivialization of $\Phi$. 
\end{proof}

\section{Conclusion}
\label{sect:conclusion}

Let us return, now, to item (ii) of our introduction's second paragraph, or in other words to the problem of understanding more general variants of the classical system $\mathbf{A}$.
For reasons outlined thereafter, the central remaining question in this direction is the following:
\begin{question}
\label{ques:conclusion}
Is it consistent with the $\mathsf{ZFC}$ axioms that $\mathrm{lim}^n\,\mathbf{A}_\lambda=0$ for all cardinals $\lambda$ and integers $n>0$?
\end{question}
\noindent The question is, at this point, a well-documented one: minor variations on it appear as \cite[Question 3]{Be17}, \cite[Question 7.10]{BHLH23}, and \cite[Question 5.7]{BeLH24}; see also \cite[Questions 7.3 and 7.1]{Ban23}, respectively, for a subinstance and more general formulation.
Moreover, as Bannister has very recently shown \cite{BanAk25}, an affirmative answer is in fact equivalent to the additivity of strong homology on all locally compact metric spaces. 
Any of these variants of Question \ref{ques:conclusion} should ultimately be regarded as asking just how additive a derived limit functor is compatible with the $\mathsf{ZFC}$ axioms.
This paper's results lend this question additional interest, both by ruling out the most optimistic scenario for its solution (namely, the extension of the result \cite[Theorem 5.1]{Be17} recorded in our introduction as \emph{Theorem} to all degrees $n>0$), and, more broadly, by identifying fundamental combinatorial possibilities which any positive solution to the question must sidestep.
Note also the relation to our introduction's item (i): if a main theme in the study of $\mathrm{lim}^n\,\mathbf{A}$ has been \emph{vertical} interactions (via the cardinal invariant $\mathfrak{d}$, for example) of ordinal combinatorics with product orderings like $({^\omega}\omega,\leq)$, it's their persistent \emph{horizontal} interactions in our analyses that seem to account for some of the difficulty of Question \ref{ques:conclusion}, and of research line (ii) more generally.
Note lastly two variations on this work's initiating question left open:
\begin{question}
Does $\mathrm{lim}^n\,\mathbf{A}=0$ for all integers $n>0$ imply that $\mathrm{lim}^n\,\mathbf{A}_\lambda=0$ for all cardinals $\lambda$ and integers $n>0$?
\end{question}
\begin{question}
Does there exist for each integer $n>0$ an accessible cardinal $\kappa$ such that $\mathrm{lim}^n\,\mathbf{A}_\kappa=0$ implies that $\mathrm{lim}^n\,\mathbf{A}_\lambda=0$ for all cardinals $\lambda\geq\kappa$?
\end{question}
\noindent Affirmative answers to either continue to figure as central scenarios for a positive solution to Question \ref{ques:conclusion}.
\bibliography{limnAkappa_bib}
\bibliographystyle{amsalpha}

\end{document}